\documentclass[11pt,a4paper,twoside]{article}

\usepackage[french,english]{babel}
\usepackage{amsfonts}
\usepackage{amsmath,amssymb,amscd}
\usepackage{mathrsfs}  
\usepackage{amsthm}
\usepackage{a4wide}   
\usepackage[T1]{fontenc} 
\usepackage{newlfont} 
\usepackage{color}
\usepackage{stackrel}

\usepackage[notref,notcite]{showkeys}    

\newcommand{\cache}[1]{}

\newcommand{\marge}[1]{\marginpar{#1}} 

\newcommand{\red}[1]{{\color{red}#1}}
\newcommand{\blue}[1]{{\color{blue}#1}}
\newcommand{\green}[1]{{\color{green}#1}}

\usepackage[
backref=page,bookmarksopen=true]{hyperref} 

\usepackage[all]{xy}  

\usepackage[applemac]{inputenc}       

\def\choixcompteur{subsection}
\newtheorem{theo}[\choixcompteur]{Theorem}

\newtheorem{prop}[\choixcompteur]{Proposition}
\newtheorem{lemm}[\choixcompteur]{Lemma}

\theoremstyle{definition}

\newtheorem{rema}[\choixcompteur]{Remark}

\newtheorem*{exem*}{Example}
\newtheorem*{exems*}{Examples}
\newtheorem*{exam*}{Example}
\newtheorem*{exams*}{Examples}
\newtheorem*{rema*}{Remark}
\newtheorem*{remas*}{Remarks}
\newtheorem*{NB}{N.B}

\theoremstyle{definition}
\newtheorem*{defi*}{Definition}
\newtheorem*{defiprop*}{Definition-Proposition}

\theoremstyle{plain}
\newtheorem*{prop*}{Proposition}
\newtheorem*{lemm*}{Lemma}
\newtheorem*{coro*}{Corollary}
\newtheorem*{theo*}{Theorem}
\newtheorem*{conj*}{Conjecture}

 \def\cdr@enoncedef{%
 \newenvironment{enonce*}[2][plain]%
 {\let\cdrenonce\relax \theoremstyle{##1}%
 \newtheorem*{cdrenonce}{##2}%
 \begin{cdrenonce}}%
 {\end{cdrenonce}}   }%

 \AtBeginDocument{%
    \cdr@enoncedef}

\def\cf{{\it cf.\/}\ }
\def\ie{{\it i.e.\/}\ }
\def\eg{{\it e.g.\/}\ }

\def\lc{{\it l.c.\/}\ }

\def\resp{{\it resp.,\/}\ }

\def\truc{\unskip\kern 3pt\penalty 500
\hbox{\vrule\vbox to 5pt{\hrule width 4pt\vfill\hrule}\vrule}\kern 3pt}

\def\vect{\overrightarrow}
\def\un{\underline}
\def\ov{\overline}

\def\parni{\par\noindent}

\def\eds{ editors}

\def\N{{\mathbb N}}    
\def\Z{{\mathbb Z}}

\def\R{{\mathbb R}}
\def\C{{\mathbb C}}

\def\A{{\mathbb A}}
\def\M{{\mathbb M}}

\newcommand{\g}[1]{\mathfrak{#1}} 

\def\qa{\alpha}     
\def\qb{\beta}
\def\qd{\delta}
\def\qe{\varepsilon}
\def\qf{\varphi}

 \def\qk{\kappa}
 \def\ql{\lambda}
\def\qm{\mu}
\def\qn{\nu}

\def\qp{\pi}
\def\qr{\rho}
\def\qs {\sigma}

\def\qx{\xi}
 \def\qz{\zeta}

\def\QD{\Delta}
\def\QF{\Phi}

\def\QL{\Lambda}

\def\sha{{\mathcal A}}   

\def\shc{{\mathcal C}}

\def\shk{{\mathcal K}}

\def\shm{{\mathcal M}}

\def\sho{{\mathcal O}}

\def\shq{{\mathcal Q}}

\def\sht{{\mathcal T}}

\def\SHC{{\mathscr C}}

\def\SHH{{\mathscr H}}
\def\SHI{{\mathscr I}}

\def\SHQ{{\mathscr Q}}

\def\SHS{{\mathscr S}}

\begin{document}

\title{On structure constants of Iwahori-Hecke algebras for Kac-Moody groups }
\author{Nicole Bardy-Panse and Guy Rousseau}

\date{September 16, 2019}

\maketitle

%





\begin{abstract} We consider the Iwahori-Hecke algebra $^I\!\SHH$ associated to an almost split Kac-Moody group $G$ (affine or not) over a nonarchimedean local field $\shk$.
It has a canonical double-coset basis $(T_{\mathbf w})_{\mathbf w\in W^+}$ indexed by a sub-semigroup $W^+$ of the affine Weyl group $W$.
The multiplication is given by structure constants $a^ {\mathbf u}_{\mathbf w,\mathbf v}\in\N=\Z_{\geq0}$ : $T_{\mathbf w}*T_{\mathbf v}=\sum_{\mathbf u\in P_{\mathbf w,\mathbf v}} a^ {\mathbf u}_{\mathbf w,\mathbf v} T_{\mathbf u}$.
A conjecture, by Bravermann, Kazhdan, Patnaik, Gaussent and the authors, tells that $a^ {\mathbf u}_{\mathbf w,\mathbf v}$ is a polynomial, with coefficients in $\N$, in the parameters $q_{i}-1,q'_{i}-1$ of $G$ over $\shk$.       
We prove this conjecture when $\mathbf w$ and $\mathbf v$ are spherical 
or, more generally, when they are said generic: this includes all cases of $\mathbf w,\mathbf v\in W^+$ if $G$ is of affine or strictly hyperbolic type.
In the split affine case (where $q_{i}=q'_{i}=q$, $\forall i$) we get a universal Iwahori-Hecke algebra with the same basis $(T_{\mathbf w})_{\mathbf w\in W^+}$ over a polynomial ring $\Z[Q]$; it specializes to $^I\!\SHH$ when one sets $Q=q$.
\end{abstract}

\setcounter{tocdepth}{1}    

\section*{Introduction}
\label{seIntro}  Let $G$ be a split, semi-simple, simply connected algebraic group over a non archimedean local field $\shk$.
 So $\shk$ is complete for a discrete, non trivial valuation with a finite residue field $\qk$.
 We write $\sho\subset\shk$ the ring of integers and $q$ the cardinality of $\qk$.
 Then $G$ is locally compact. In this situation,  Nagayoshi Iwahori and Hideya Matsumoto in \cite{IM65}, introduced an open compact subgroup $K_{I}$ of $G$, now known as an Iwahori subgroup.
 If $N$ is the normalizer of a suitable split maximal torus $T\simeq (\shk^*)^n$, then $(K_{I},N)$ is a BN pair.
 The Iwahori-Hecke algebra of $G$ is the algebra ${^I\!}\SHH_{R}={^I\!}\SHH_{R}(G,K_{I})$ of locally constant, compactly supported functions on $G$, with values in a ring $R$, that are bi-invariant by the left and right actions of $K_{I}$.
 The multiplication is given by the convolution product.
 
 \par If $H\simeq(\sho^*)^n$ is the maximal compact subgroup of $T$, then $H\subset K_{I}$ and $W=N/H$ is the affine Weyl group.
 One has the Bruhat decomposition $G=K_{I}.W.K_{I}=\sqcup_{\mathbf w\in W} K_{I}.\mathbf w.K_{I}$.
 If one considers the characteristic function $T_{\mathbf w}$ of $K_{I}.\mathbf w.K_{I}$, we get a basis of ${^I\!}\SHH_{R}$: ${^I\!}\SHH_{R}=\oplus_{\mathbf w\in W} R.T_{\mathbf w}$.
 The convolution product is given by $T_{\mathbf w}*T_{\mathbf v}=\sum_{\mathbf u\in P_{\mathbf w,\mathbf v}} a^ {\mathbf u}_{\mathbf w,\mathbf v} T_{\mathbf u}$, with $P_{\mathbf w,\mathbf v}$ a finite subset  of $W$.
 The numbers $a^ {\mathbf u}_{\mathbf w,\mathbf v}\in R$ are the structure constants of ${^I\!}\SHH_{R}$. The unit is $1=T_{e}$.

\par Iwahori and Matsumoto gave a precise (and now classical) definition of $\SHH_{R}$ by generators and relations.
The group $W$ is an infinite Coxeter group generated by $\{r_{0},\ldots,r_{n}\}$.
Then $\SHH_{R}$ is generated by $\{T_{r_{0}},\ldots,T_{r_{n}}\}$ with relations $T_{r_{i}}^2=q.1+(q-1).T_{r_{i}}$ and $T_{r_{i}}*T_{r_{j}}*T_{r_{i}}*\cdots=T_{r_{j}}*T_{r_{i}}*T_{r_{j}}*\cdots$ (with $m_{i,j}$ factors on each side) for $i\neq j$, if $m_{i,j}$ is the finite order of $r_{i}r_{j}$.
For $\mathbf w=r_{i_{1}}.\ldots.r_{i_{s}}$ a reduced expression in $W$, one has $T_{\mathbf w}=T_{r_{i_{1}}}*\cdots*T_{r_{i_{s}}}$.
In a Coxeter group one knows the rules to get (using the Coxeter relations between the $r_{i}$) a reduced expression from a non reduced expression (\eg the product of two reduced expressions $\mathbf w=r_{i_{1}}.\ldots.r_{i_{s}}$ and $\mathbf v=r_{j_{1}}.\ldots.r_{j_{t}}$).
So one deduces easily (using the above relations between the $T_{r_{i}}$) that each structure constant $a^ {\mathbf u}_{\mathbf w,\mathbf v}$ (for $\mathbf u, \mathbf v,\mathbf w \in W$) is in $\Z[q]$.
More precisely it is a polynomial in $q-1$ with coefficients in $\N=\Z_{\geq0}$.
This polynomial depends only on $\mathbf u, \mathbf v,\mathbf w$ and $W$.

\par So one has a universal description of ${^I\!}\SHH_{\Z}$ as a $\Z[q]-$algebra, depending only on $W$.

\par There are various generalizations of the above situation.
First one may replace $G$ by a general reductive group over $\shk$, isotropic but eventually non split.
Then one has to consider the relative affine Weyl group $W$, which is a Coxeter group.
One may still define a compact, open Iwahori subgroup  $K_{I}$ and there is a Bruhat decomposition $G=K_{I}.W.K_{I}$.
Now the description of ${^I\!}\SHH_{R}$ involves parameters $q_{i}$ (satisfying $T_{r_{i}}^2=q_{i}.1+(q_{i}-1).T_{r_{i}}$) which are eventually different from $q$.
This gives the Iwahori-Hecke algebra with unequal parameters.
There is a pleasant description of ${^I\!}\SHH_{R}$ using the Bruhat-Tits building associated to the BN pair $(K_{I},N)$, see \eg \cite{P06}.

For now more than twenty years, there is an increasing interest in the study of Kac-Moody groups over local fields, see essentially the works of Bravermann, Garland, Kapranov, Kazhdan, Patnaik, Gaussent and the authors: \eg \cite{Ga95}, \cite{GaG95}, \cite{Kap01}, \cite{BrK11}, \cite{BrK14}, \cite{BrGKP14}, \cite{BrKP14}, \cite{GR13}, \cite{BPGR16}, \cite{BPGR17}.
 It has been possible to define and study for Kac-Moody groups (supposed at first affine) the spherical Hecke algebra, the Iwahori-Hecke algebra, the Satake isomorphism, \ldots.
This is also closely related to more abstract works on Hecke algebras by Cherednik and Macdonald, \eg  \cite{Che92}, \cite{Che95}, \cite{Ma03}.

\par We are mainly interested in Iwahori-Hecke algebras for Kac-Moody groups over local fields.
They were introduced and described by Bravermann, Kazhdan and Patnaik in the affine case \cite{BrKP14} and then in general by Gaussent and the authors \cite{BPGR16}.
So let us consider a Kac-Moody group $G$ (affine or not) over the local field $\shk$.
We suppose it split (as defined by Tits \cite{T87}) or more generally almost split \cite{Re02}. Let us choose also a maximal split subtorus.
To this situation is associated an affine (relative) Weyl group $W$ and an Iwahori subgroup $K_{I}$ (defined up to conjugacy by $W$), see \ref{1.3} (5) and (7) below.
This group $W$ is not a Coxeter group but may be described as a semi-direct product $W=W^v\ltimes Y$, where $W^v$ is a Coxeter group, the relative Weyl group, and $Y$ is (essentially) the cocharacter group of the torus.

\par Unfortunately the Bruhat decomposition ``$G=K_{I}.W.K_{I}$'' fails to be true (even in the untwisted affine case, \ie for loop groups).
One has to consider the sub-semigroup $W^+=W^v\ltimes Y^+$ (\resp $W^{+g}=W^v\ltimes Y^{+g}$) of $W$, where $Y^+$ (\resp $Y^{+g}$) is the intersection of $Y$ with the Tits cone $\sht$ (\resp  with a cone $\sht^\circ\cup V_{0}\subset\sht$, where $\sht^\circ$ is the open Tits cone) in $V=Y\otimes_{\Z}\R$ (see \ref{1.2}, \ref{1.4}, and \ref{1.11} below).
Then $G^+=K_{I}.W^+.K_{I}$ (\resp $G^{+g}=K_{I}.W^{+g}.K_{I}\subset G^+$) is a sub-semigroup of $G$: the Kac-Moody-Tits semigroup (\resp  the generic Kac-Moody-Tits semigroup).
We may consider the characteristic functions $T_{\mathbf w}$ of the double cosets $K_{I}.\mathbf w.K_{I}$ and one proves in \cite{BPGR16} that:

\medskip
\par The space ${^I\!}\SHH_{R}$ (\resp ${^I\!}\SHH^g_{R}$) of $R-$valued functions with finite support on $K_{I}\backslash G^ {+}/K_{I}$ (\resp $K_{I}\backslash G^ {+g}/K_{I}$) is naturally endowed with a structure of algebra (see \ref{s2}).
We get thus the Iwahori-Hecke algebra ${^I\!}\SHH_{R}=\oplus_{\mathbf w\in W^+} R.T_{\mathbf w}$ (\resp the generic Iwahori-Hecke algebra ${^I\!}\SHH^g_{R}=\oplus_{\mathbf w\in W^{+g}} R.T_{\mathbf w}$). 
 The product is given by structure constants $a^ {\mathbf u}_{\mathbf w,\mathbf v}\in\N=\Z_{\geq0}$: $T_{\mathbf w}*T_{\mathbf v}=\sum_{\mathbf u\in P_{\mathbf w,\mathbf v}} a^ {\mathbf u}_{\mathbf w,\mathbf v} T_{\mathbf u}$.

\begin{enonce*}[plain]{Conjecture 1} \cite[2.5]{BPGR16} Each $a^ {\mathbf u}_{\mathbf w,\mathbf v}$ is a polynomial, with coefficients in $\N=\Z_{\geq0}$, in the parameters $q_{i}-1,q'_{i}-1$ of the situation, see \ref{1.3}.6 below.
This polynomial depends only on the affine Weyl group $W$ acting on the apartment $\A$ and on $\mathbf w,\mathbf v,\mathbf u\in W^+$.
\end{enonce*}

\par One may consider that this is a translation of the following question of Braverman, Kazhdan and Patnaik :

\begin{enonce*}[plain]{Question} \cite[end of 1.2.4]{BrKP14} Has the algebra ${^I\!}\SHH_{\C}$ a purely algebraic or combinatorial description with respect to the coset basis  $(T_{\mathbf w})_{\mathbf w\in W^+}$ ?
\end{enonce*}

\par But a more precise formulation of this question is as follows :

\begin{enonce*}[plain]{Conjecture 2} The algebra ${^I\!}\SHH_{\Z}$ (or ${^I\!}\SHH^g_{\Z}$) is the specialization of an algebra ${^I\!}\SHH_{\Z[\SHQ]}$ (or ${^I\!}\SHH^g_{\Z[\SHQ]}$) with the same basis $(T_{\mathbf w})_{\mathbf w\in W^+}$ (or $(T_{\mathbf w})_{\mathbf w\in W^{+g}}$) over $\Z[\SHQ]$.
Here $\SHQ$ is a set of indeterminates $Q_{i},Q'_{i}$ (with some equalities between them, see \ref{1.3}.6 below) and the specialization is given by $Q_{i}\mapsto q_{i},Q'_{i}\mapsto q'_{i}, \forall i\in I$.
The algebra  ${^I\!}\SHH_{\Z[\SHQ]}$ (or ${^I\!}\SHH^g_{\Z[\SHQ]}$)  depends only on the affine Weyl group $W$ acting on the apartment $\A$. 
\end{enonce*}

\par Let us consider the split case: $G$ is a split Kac-Moody group, all parameters $q_{i},q'_{i}$ are equal to $q=\vert\qk\vert$ and all indeterminates $Q_{i},Q'_{i}$ are equal to a single indeterminate $Q$. 
Then the conjecture 1 has already been proved by Gaussent and the authors \cite[6.7]{BPGR16} and independently by Muthiah \cite{Mu15} if, moreover, $G$ is untwisted affine.
Actually the same proof gives also conjecture 2, see \ref{1.3}.7 below.

 In the general (non split) case weakened versions were obtained in \cite{BPGR16}: the $a^ {\mathbf u}_{\mathbf w,\mathbf v}$ are Laurent polynomials in the $q_{i},q'_{i}$ [\lc 6.7]; they are true polynomials if $\mathbf w,\mathbf v \in W^v\ltimes(Y\cap\sht^\circ)$ and $\mathbf v$ is ``regular'' [\lc 3.8].

\par In this article, we prove the conjecture 1 when $\mathbf w$ and $\mathbf v$ are in $W^ {+g}$ (see \ref{sc8}).
We remark also that $W^+=W^ {+g}$ in the affine case (twisted or not) or the strictly hyperbolic case, even if $G$ is not split. 
This is a first step towards the description of an abstract algebra ${^I\!}\SHH_{\Z[\SHQ]}$ (\resp ${^I\!}\SHH^g_{\Z[\SHQ]}$) over $\Z[\SHQ]$ in the affine  (or strictly hyperbolic) case (\resp in the general case).

\par One should mention here that one may give a more precise description of the Iwahori-Hecke algebra using a Bernstein-Lusztig presentation (see \cite{GaG95}, \cite{BrKP14} and \cite{BPGR16}).
But this description is given in a new basis and the coefficients of the change of basis matrix are Laurent polynomials in the parameters $q_{i},q'_{i}$.
So this description is not sufficient to prove the conjecture.

\par Actually this article is written in a more general framework explained in Section \ref{s1}: as in \cite{BPGR16}, we work with an abstract masure $\SHI$ and we take $G$ to be a strongly transitive group of vectorially-Weyl automorphisms of $\SHI$.
In Section \ref{pr} we gather the additional technical tools (\eg decorated Hecke paths) needed to improve the results of \cite[Section 3]{BPGR16}. 
We get our main results about $a^ {\mathbf u}_{\mathbf w,\mathbf v}$ in Section \ref{sc}: we deal with the cases $\mathbf w,\mathbf v$ spherical.
In Section \ref{s4} we deal with the remaining cases where $\mathbf w,\mathbf v$  are in $W^ {+g}$, \ie when $\mathbf w,\mathbf v$ are said generic.

\section{General framework}\label{s1}

\subsection{Vectorial data}\label{1.1}  We consider a quadruple $(V,W^v,(\qa_i)_{i\in I}, (\qa^\vee_i)_{i\in I})$ where $V$ is a finite dimensional real vector space, $W^v$ a subgroup of $GL(V)$ (the vectorial Weyl group), $I$ a finite set, $(\qa^\vee_i)_{i\in I}$ a {free} family in $V$ and $(\qa_i)_{i\in I}$ a free family in the dual $V^*$.
 We ask these data to satisfy the conditions of \cite[1.1]{R11}.
  In particular, the formula $r_i(v)=v-\qa_i(v)\qa_i^\vee$ defines a linear involution in $V$ which is an element in $W^v$ and $(W^v,\{r_i\mid i\in I\})$ is a Coxeter system.

  \par To be more concrete, we consider the Kac-Moody case of [\lc; 1.2]: the matrix $\M=(\qa_j(\qa_i^\vee))_{i,j\in I}$ is a generalized Cartan matrix.
  Then $W^v$ is the Weyl group of the corresponding Kac-Moody Lie algebra $\g g_\M$ and the associated real root system is
$$
\QF=\{w(\qa_i)\mid w\in W^v,i\in I\}\subset Q=\bigoplus_{i\in I}\,\Z.\qa_i.
$$ We set $\QF^\pm{}=\QF\cap Q^\pm{}$ where $Q^\pm{}=\pm{}(\bigoplus_{i\in I}\,(\Z_{\geq 0}).\qa_i)$ and $Q^\vee=(\bigoplus_{i\in I}\,\Z.\qa_i^\vee)$, $Q^\vee_\pm{}=\pm{}(\bigoplus_{i\in I}\,(\Z_{\geq 0}).\qa_i^\vee)$.
   We have  $\QF=\QF^+\cup\QF^-$ and, for $\qa=w(\qa_i)\in\QF$, $r_\qa=w.r_i.w^{-1}$ and $r_\qa(v)=v-\qa(v)\qa^\vee$, where the coroot $\qa^\vee=w(\qa_i^\vee)$ depends only on $\qa$.

\par The set $\QF$ is an (abstract, reduced) real root system in the sense of \cite{MP89}, \cite{MP95} or \cite{Ba96}.
We shall sometimes also use the set $\QD=\QF\cup\QD^+_{im}\cup\QD^-_{im}$ of all roots (with $-\QD^-_{im}=\QD^+_{im}\subset Q^+$,  $W^v-$stable) defined in \cite{K90}.
 It is an (abstract, reduced) root system in the sense of \cite{Ba96}.

  \par The {\it fundamental positive chamber} is $C^v_f=\{v\in V\mid\qa_i(v)>0,\forall i\in I\}$.
   Its closure $\overline{C^v_f}$ is the disjoint union of the vectorial faces $F^v(J)=\{v\in V\mid\qa_i(v)=0,\forall i\in J,\qa_i(v)>0,\forall i\in I\setminus J\}$ for $J\subset I$. 
   We set $V_0 = F^v(I)$.
    The positive (resp. negative) vectorial faces are the sets $w.F^v(J)$ (resp. $-w.F^v(J)$) for $w\in W^v$ and $J\subset I$.
    The support of such a face is the vector space it generates.
    The set $J$ or the face $w.F^v(J)$ or an element of this face is called {\it spherical} if the group $W^v(J)$ generated by $\{r_i\mid i\in J\}$ is finite.
    An element of a vectorial chamber $\pm w.C^v_f$ is called {\it regular}.

    \par The {\it Tits cone}  $\sht$ (\resp its interior $\sht^\circ$) is the (disjoint) union of the positive (\resp and spherical) vectorial faces. It is a $W^v-$stable convex cone in $V$.
One has $\sht=\sht^\circ=V$ (\resp $V_{0} \subset \sht \setminus \sht^\circ$) in the classical (\resp non classical\blue{)} case, \ie when $W^v$ is finite (\resp infinite).

  \par 
  We say that $\A^v=(V,W^v)$ is a {\it vectorial apartment}.

\subsection{The model apartment}\label{1.2} As in \cite[1.4]{R11} the model apartment $\A$ is $V$ considered as an affine space and endowed with a family $\shm$ of walls. 
 These walls  are affine hyperplanes directed by $\ker(\qa)$ for $\qa\in\QF$.
 More precisely, they may be written $M(\qa,k)=\{v\in V\mid\qa(v)+k=0\}$, for $\qa\in\QF$ and $k\in\R$.

 \par We ask this apartment to be {\bf semi-discrete} and the origin $0$ to be {\bf special}.
  This means that these walls are the hyperplanes $M(\qa,k)=\{v\in V\mid\qa(v)+k=0\}$ for $\qa\in\QF$ and $k\in\QL_\qa,$ with $\QL_\qa=k_\qa.\Z$ a non trivial discrete subgroup of $\R$. 
Using   \cite[Lemma 1.3]{GR13} (\ie replacing $\QF$ by another system $\QF_1$) we may (and shall) assume that $\QL_\qa=\Z, \forall\qa\in\QF$.

  \par For $\qa=w(\qa_i)\in\QF$, $k\in\Z$ and $M=M(\qa,k)$, the reflection $r_{\qa,k}=r_M$ with respect to $M$ is the affine involution of $\A$ with fixed points the wall $M$ and associated linear involution $r_\qa$.
   The affine Weyl group $W^a$ is the group generated by the reflections $r_M$ for $M\in \shm$; we assume that $W^a$ stabilizes $\shm$.
We know that $W^a=W^v\ltimes Q^\vee$ and we write $W^a_\R=W^v\ltimes V$; here $Q^\vee$ and $V$ have to be understood as groups of translations.

   \par An automorphism of $\A$ is an affine bijection $\qf:\A\to\A$ stabilizing the set of pairs $(M,\qa^\vee)$ of a wall $M$ and the coroot associated with $\qa\in\QF$ such that $M=M(\qa,k)$, $k\in\Z$. The group $Aut(\A)$ of these automorphisms contains $W^a$ and normalizes it.
We consider also the group $Aut^W_\R(\A)=\{\qf\in Aut(\A)\mid\vect{\qf}\in W^v\}=Aut(\A)\cap W^a_\R$.

   \par For $\qa\in\QF$ and $k\in\R$, $D(\qa,k)=\{v\in V\mid\qa(v)+k\geq 0\}$ is an half-space, it is called an {\it half-apartment} if $k\in\Z$. We write  $D(\alpha,\infty) = \mathbb A$.

The Tits cone $\mathcal T$ 
and its interior $\mathcal T^o$ are convex and $W^v-$stable cones, therefore, we can define three $W^v-$invariant preorder relations  on $\mathbb A$: 
$$
x\leq y\;\Leftrightarrow\; y-x\in\mathcal T
; \quad x\stackrel{o}{<} y\;\Leftrightarrow\; y-x\in\mathcal T^o
; \quad x\stackrel{o}{\leq} y\;\Leftrightarrow\; y-x\in\mathcal T^o\cup V_{0}.
$$
 If $W^v$ has no  fixed point in $V\setminus\{0\}$ (\ie $V_{0}=\{0\}$) and no finite factor, then they are orders; but, in general, they are not.

\subsection{Faces, sectors}
\label{suse:Faces}

 The faces in $\mathbb A$ are associated to the above systems of walls
and half-apartments. As in \cite{BrT72}, they
are no longer subsets of $\mathbb A$, but filters of subsets of $\mathbb A$. For the definition of that notion and its properties, we refer to \cite{BrT72} or \cite{GR08}.

If $F$ is a subset of $\mathbb A$ containing an element $x$ in its closure,
the germ of $F$ in $x$ is the filter $\mathrm{germ}_x(F)$ consisting of all subsets of $\mathbb A$ which contain intersections of $F$ and neighbourhoods of $x$. In particular, if $x\neq y\in \mathbb A$, we denote the germ in $x$ of the segment $[x,y]$ (resp. of the interval $]x,y]$) by $[x,y)$ (resp. $]x,y)$).

\par For $y≠x$, the segment germ $[x,y)$ is called of sign $\pm$ if $y-x\in\pm\sht$.
The segment $[x,y]$ (or the segment germ $[x,y)$ or the ray $[x,y($ with origin $x$ containing $y$) is called {\it preordered}  if $x\leq y$ or $y\leq x$ and {\it generic}  if $x\stackrel{o}{<} y$ or $y\stackrel{o}{<} x$.

Given $F$ a filter of subsets of $\mathbb A$, its {\it strict enclosure} $cl_{\mathbb A}(F)$ (resp. {\it closure} $\overline F$) is the filter made of  the subsets of $\mathbb A$ containing an element of $F$ of the shape $\cap_{\alpha\in\Delta}D(\alpha,k_\alpha)$, where $k_\alpha\in {\Z}\cup\{\infty\}$ (resp. containing the closure $\overline S$ of some $S\in F$). 
One considers also the (larger) {\it enclosure} $cl_\A^\#(F)$ of \cite[3.6.1]{R13} (introduced in \cite{Ch10}, \cite {Ch11} and well studied in \cite{He17a},  see also \cite{Heb18}).
It is  the filter made of  the subsets of $\mathbb A$ containing an element of $F$ of the shape $\cap_{\alpha\in\Psi}D(\alpha,k_\alpha)$, with $\Psi\subset\QF$ finite and  $k_\alpha\in {\Z}$  (\ie a finite intersection of half apartments).

\medskip

A {\it local face} $F$ in the apartment $\mathbb A$ is associated
 to a point $x\in \mathbb A$, its vertex, and a  vectorial face $F^v$ in $V$, its direction. It is defined as $F=germ_x(x+F^v)$ and we denote it by $F=F^\ell(x,F^v)$.
 Its closure is $\overline{F^\ell}(x,F^v)=germ_x(x+\overline{F^v})$
There is an order on the local faces: the assertions ``$F$ is a face of $F'$ '',
``$F'$ covers $F$ '' and ``$F\leq F'$ '' are by definition equivalent to
$F\subset\overline{F'}$.
 The dimension of a local face $F$ is the smallest dimension of an affine space generated by some $S\in F$.
  The (unique) such affine space $E$ of minimal dimension is the support of $F$; if $F=F^\ell(x,F^v)$, $supp(F)=x+supp(F^v)$.
 A local face $F=F^\ell(x,F^v)$ is spherical if the direction of its support meets the open Tits cone (\ie when $F^v$ is spherical), then its  pointwise stabilizer $W_F$ in $W^a$ or $W^a_{\R}$ is finite and fixes $x$.
 
 \par We shall actually here speak only of local faces, and sometimes forget the word local or write $F=F(x,F^v)$.

\medskip

 A {\it local chamber}  is a maximal local face, \ie a local face $F^\ell(x,\pm w.C^v_f)$ for $x\in\A$ and $w\in W^v$.
 The {\it fundamental local positive (\resp negative)} chamber is $C_0^+=germ_0(C^v_f)$ (\resp $C_0^-=germ_0(-C^v_f)$).

A {\it (local) panel} is a spherical local face maximal among local faces which are not chambers, or, equivalently, a spherical face of dimension $n-1$. Its support is a wall.

\medskip
 A {\it sector} in $\mathbb A$ is a $V-$translate $\mathfrak s=x+C^v$ of a vectorial chamber
$C^v=\pm w.C^v_f$, $w \in W^v$. The point $x$ is its {\it base point} and $C^v$ its  {\it direction}.  Two sectors have the same direction if, and only if, they are conjugate
by $V-$translation,
 and if, and only if, their intersection contains another sector.

 The {\it sector-germ} of a sector $\mathfrak s=x+C^v$ in $\mathbb A$ is the filter $\mathfrak S$ of
subsets of~$\mathbb A$ consisting of the sets containing a $V-$translate of $\mathfrak s$, it is well
determined by the direction $C^v$. So, the set of
translation classes of sectors in $\mathbb A$, the set of vectorial chambers in $V$ and
 the set of sector-germs in $\mathbb A$ are in canonical bijection.

 A {\it sector-face} in $\mathbb A$ is a $V-$translate $\mathfrak f=x+F^v$ of a vectorial face
$F^v=\pm w.F^v(J)$. The sector-face-germ of $\mathfrak f$ is the filter $\mathfrak F$ of
subsets containing a translate $\mathfrak f'$ of $\mathfrak f$ by an element of $F^v$ ({\it i.e.} $\mathfrak
f'\subset \mathfrak f$). If $F^v$ is spherical, then $\mathfrak f$ and $\mathfrak F$ are also called
spherical. The sign of $\mathfrak f$ and $\mathfrak F$ is the sign of $F^v$.

 \subsection{The Masure}\label{1.3}

 In this section, we recall the definition and some properties of a masure given by Guy Rousseau in \cite{R11} and simplified by Auguste H\'ebert \cite{He17a}.

\medskip
\parni{\bf 1)} An apartment of type $\mathbb A$ is a set $A$ endowed with a set $Isom^W\!(\mathbb A,A)$ of bijections (called Weyl-isomorphisms) such that, if $f_0\in Isom^W\!(\mathbb A,A)$, then $f\in Isom^W\!(\mathbb A,A)$ if, and only if, there exists $w\in W^a$ satisfying $f = f_0\circ w$.
An isomorphism (resp. a Weyl-isomorphism, a vectorially-Weyl isomorphism) between two apartments $\varphi :A\to A'$ is a bijection such that, for any $f\in Isom^W\!(\mathbb A,A)$, $f'\in Isom^W\!(\mathbb A,A')$, $f'^{-1}\circ\qf\circ f\in Aut(\A)$ (resp. $\in W^a$, $\in Aut^W_\R(\A)$); the group of these isomorphisms is written $Isom(A,A')$ (resp. $Isom^W(A,A')$, $Isom^W_\R(A,A')$).
 As the filters in $\A$ defined in \ref{suse:Faces} above (\eg local faces, sectors, walls,..) are permuted by $Aut(\A)$, they are well defined in any apartment of type $\A$ and exchanged by any isomorphism.
\medskip
\par  A {\it masure} (formerly called an {\it ordered affine hovel}) of type $\mathbb A$ is a set $\SHI$ endowed with a covering $\mathcal A$ of subsets  called apartments, each endowed with some structure of an apartment of type $\A$.
 We  recall here the simplification and improvement of the original definition given by Auguste Hébert in \cite{He17a}: these data have to satisfy the following two axioms :

\medskip
\par (MA ii) If two apartments $A,A'$ are such that $A\cap A'$ contains a generic ray, then $A\cap A'$ is a finite intersection of half-apartments (\ie $A\cap A'=cl_A^\#(A\cap A')$) and there exists a Weyl isomorphism $\qf:A\to A'$ fixing $A\cap A'$.

\par (MA iii) If $\g R$ is the germ of a splayed chimney and if $F$ is a local face or a germ of a chimney, then there exists an apartment containing $\g R$ and $F$.

\medskip
\par Actually a filter or subset in $\SHI$ is called a preordered (or generic) segment (or  segment germ), a local face, a spherical sector face or a spherical sector face germ if it is included in some apartment $A$ and is called like that in $A$.
We do not recall here what is (a germ of) a (splayed) chimney; it contains (the germ of) a (spherical) sector face.
We shall actually use (MA iii) uniquely through its consequence b) below.

\par In the affine case the hypothesis ``$A\cap A'$ contains a generic ray'' is useless in (MA ii).

\medskip
We list now some of the properties of masures we shall use.

\par {\bf a)} If $F$ is a point, a preordered segment,  a local face or a spherical sector face in an apartment $A$ and if $A'$ is another apartment containing $F$, then $A\cap A'$ contains the
enclosure {$cl_A^\#(F)$} of $F$ and there exists a Weyl-isomorphism from $A$ onto $A'$ fixing $cl_A^\#(F)$, see \cite[5.11]{He17a} or \cite[4.4.10]{Heb18}.
Hence any isomorphism from $A$ onto $A'$ fixing $F$ fixes $\overline F$ (and even $cl_A^\#(F)\cap supp(F)$).

\par More generally the intersection of two apartments $A,A'$ is always closed (in $A$ and $A'$), see \cite[3.9]{He17a} or \cite[4.2.17]{Heb18}.

\par {\bf b)} If $\mathfrak F$ is  the germ of a spherical sector face and if $F$ is a {local} face or a germ of a  sector face, then there exists an apartment that contains $\mathfrak F$ and $F$.

\par {\bf c)}  If two apartments $A,A'$ contain $\mathfrak F$ and $F$ as in {\bf b)}, then their intersection contains  {$cl_A^\#(\mathfrak F\cup F)$} and there exists a Weyl-isomorphism from $A$ onto $A'$ fixing {$cl_A^\#(\mathfrak F\cup F)$}.

\par {\bf d)} We consider the relations $≤$, $\stackrel{o}{<}$ and $\stackrel{o}{\leq}$ on $\SHI$ defined as follows:
$$x≤y \text{ (\resp } x\stackrel{o}{<}y,x\stackrel{o}{\leq}y) \iff \exists A\in\sha \text{ such that }x,y\in A\text{ and } x≤_A y  \text{ (resp. } x\stackrel{o}{<}_Ay,x\stackrel{o}{\leq}_Ay) $$

\par Then $\leq$ (\resp $\stackrel{o}{<}$, $\stackrel{o}{\leq}$) is a well defined preorder relation, in particular transitive; it is called the {\it Tits preorder} (\resp {\it Tits open preorder, large Tits  open preorder}), see \cite{He17a}.

\par {\bf e)}  We ask here $\SHI$ to be thick of {\bf finite thickness}: the number of local chambers  covering a given (local) panel in a wall has to be finite $\geq 3$.
     This number is the same for any panel $F$ in a given wall $M$ \cite[2.9]{R11}; we denote it by $1+q_M=1+q_{F}$.

\par {\bf f)}  An automorphism (resp. a Weyl-automorphism, a vectorially-Weyl automorphism) of $\SHI$ is a bijection $\qf:\SHI\to\SHI$ such that $A\in\sha\iff \qf(A)\in\sha$ and then $\qf\vert_A:A\to\qf(A)$ is an isomorphism (resp. a Weyl-isomorphism, a vectorially-Weyl isomorphism).
{We write $Aut(\SHI)$ (resp. $Aut^W(\SHI)$, $Aut^W_\R(\SHI)$) the group of these automorphisms.}
\medskip
\parni{\bf 2)} For $x\in\SHI$, the set $\sht^+_x\SHI$ (resp. $\sht^-_x\SHI$) of segment germs $[x,y)$ for $y>x$ (resp. $y<x$) may be considered as a building, the {\it positive} (resp. {\it negative}) {\it tangent building}. The corresponding faces are the local faces of positive (resp. negative) direction and vertex $x$. 
For such a local face $F$, we write sometimes $[x,y)\in F$ if $]x,y)\subset F$.
The associated Weyl group is $W^v$.
 If the $W-$distance (calculated in $\sht^\pm_x\SHI$) of two local chambers is $d^W(C_x,C'_x)=w\in W^v$, to any reduced decomposition $w=r_{i_1}\cdots r_{i_n}$ corresponds a unique minimal gallery from $C_x$ to $C'_x$ of type $(i_1,\cdots,i_n)$.
 
 \par The buildings $\sht^+_x\SHI$ and $\sht^-_x\SHI$ are actually twinned. The codistance $d^{*W}(C_x,C'_x)$ of two opposite sign chambers $C_x$ and $C'_x$ is the $W-$distance $d^W(C_x, op C'_x)$, where $op C'_x$ denotes the opposite chamber to $C'_x$ in an apartment containing $C_x$ and $C'_x$.
 Similarly two segment germs $\eta\in \sht^+_x\SHI$ and $\zeta\in \sht^-_x\SHI$ are said opposite  if they are in a same apartment $A$ and opposite in this apartment (\ie in the same line, with opposite directions).

\begin{enonce*}[plain]{3) Lemma} \cite[2.9]{R11} Let $D$ be an half-apartment in $\SHI$ and $M=\partial D$ its wall (\ie its boundary).
 One considers a panel $F$ in $M$ and a local chamber $C$ in $\SHI$ covering $F$.
 Then there is an apartment containing $D$ and $C$.
 \end{enonce*}

\parni{\bf 4)}  We assume that $\SHI$ has a strongly transitive group of automorphisms $G$,  \ie {1.a and 1.c above (after replacing $cl_A^\#$ by $cl_A$) are satisfied by isomorphisms induced  by elements of $G$, \cf \cite[4.10]{R13} and \cite[4.7]{CiMR17}.}
 
  We choose in $\SHI$ a fundamental apartment which we identify with $\A$.
   As $G$ is strongly transitive, the apartments of $\SHI$ are the sets $g.\A$ for $g\in G$. The stabilizer $N$ of $\A$ in $G$ induces a group $W=\qn(N)\subset Aut(\A)$ of affine automorphisms of $\A$ which permutes the walls, local faces, sectors, sector-faces... and contains the affine Weyl group $W^a=W^v\ltimes Q^\vee$ \cite[4.13.1]{R13}.

\par   We denote the  stabilizer of $0\in\A$ in $G$ by $K$ and the pointwise stabilizer (or fixer) of $C_0^+$ (\resp $C_{0}^-$) by $K_I=K_{I}^+$ (\resp $K_{I}^-$). This group $K_I$ is called the {\it Iwahori subgroup}.

\medskip
\parni{\bf 5)}  We ask  $W=\qn(N)$ to be  {\bf vectorially-Weyl} for its action on the vectorial faces.
      This means that the associated linear map $\vect w$ of any $w\in\qn(N)$ is in $W^v$.
      As $\qn(N)$ contains $W^a$ and stabilizes $\shm$, we have $W=\qn(N)=W^v\ltimes Y$, where $W^v$ fixes the origin $0$ of $\A$ and $Y$ is a group of translations such that:
  \quad    $Q^\vee\subset Y\subset P^\vee=\{v\in V\mid\qa(v)\in\Z,\forall\qa\in\QF\}$.
 An element $\mathbf{w}\in W$ will often be written $\mathbf{w}=\ql.w$, with $\ql\in Y$ and $w\in W^v$.

  \par We ask $Y$ to be {\bf discrete} in $V$. This is clearly satisfied if $\QF$ generates $V^*$ \ie $(\qa_i)_{i\in I}$ is a basis of $V^*$.
  
  \cache{\par We write $Z=\qn^ {-1}(Y)$ and $Z_{0}=\ker(\qn)$.\marge{utile?}}

\medskip
\parni{\bf 6)} Note that there is only a finite number of constants $q_M$ as in the definition of thickness. Indeed, we must have $q_{wM}=q_M$, $\forall w\in\qn(N)$ and $w.M(\qa,k)=M(w(\qa),k),\forall w\in W^v$. So now, fix $i\in I$, as $\qa_i(\qa_i^\vee)=2$ the translation by $\qa_i^\vee$ permutes the walls $M=M(\qa_i,k)$ (for  $k\in \Z$) with two orbits.
  So, $Q^\vee\subset W^a$ has at most two orbits in the set of the constants $q_{M(\qa_i,k)}$ : one containing the $q_i=q_{M(\qa_i,0)}$ and the other containing the $q_i'=q_{M(\qa_i,\pm{}1)}$.
  Hence, the number of (possibly) different $q_M$ is at most $2.\vert I\vert$. We denote this set of parameters by $\shq=\{q_i,q'_i\mid i\in I\}$.
  
\par  In \cite[1.4.5]{BPGR16} one proves the following further equalities:  $q_{i}= q_{i}'$ if $\qa_{i}(Y)=\Z$ and $q_{i}=q'_{i}=q_{j}=q'_{j}$ if $\qa_{i}(\qa_{j}^\vee)=\qa_{j}(\qa_{i}^\vee)=-1$.

We consider also the polynomial algebra $\Z[\SHQ]$, where $\SHQ$ is the set $\SHQ=\{Q_{i},Q'_{i} \mid i\in I\}$ of indeterminates, satisfying the same equalities: $Q_{i}= Q_{i}'$ if $\qa_{i}(Y)=\Z$ and $Q_{i}=Q'_{i}=Q_{j}=Q'_{j}$ if $\qa_{i}(\qa_{j}^\vee)=\qa_{j}(\qa_{i}^\vee)=-1$.
See \cite[6.1]{BPGR16} where $Q_{i}=\qs_{i}^2, Q'_{i}=(\qs'_{i})^2$.

\cache{\par If $\qa_i(\qa_j^\vee)$ is odd for some $j\in I$, the translation by $\qa_j^\vee$ exchanges the two walls $M(\qa_i,0)$ and $M(\qa_i,-\qa_i(\qa_j^\vee))$; so $q_i=q'_i$.
If $\qa_i(\qa_j^\vee)=\qa_j(\qa_i^\vee)=-1$, one knows that the element $r_ir_jr_i$ of $W^v(\{i,j\})$ exchanges $\qa_i$ and $-\qa_j$, so $q_i=q'_i=q_j=q'_j$.}

 \medskip
  \par\noindent{\bf 7) Examples.} The main examples of all the above situation are provided by the  Kac-Moody theory, as already indicated in the introduction.
  More precisely let $G$ be an almost split Kac-Moody group over a non archimedean complete field $\shk$.
  We suppose moreover the valuation of $\shk$ discrete and its residue field $\qk$ perfect.
Then there is a masure $\SHI$ on which $G$ acts strongly transitively by vectorially Weyl automorphisms.
If $\shk$ is a local field (\ie $\qk$ is finite), then we are in the situation described above.
This is the main result of \cite{Ch10}, \cite{Ch11} and \cite{R13}.

\par When $G$ is actually split, this result was known previously by \cite{GR13} and \cite{R12}.
And in this case all the constants $q_{M}, q_{i}, q'_{i}$ are equal to the cardinality $q$ of the residue field $\qk$.

\par We gave in \cite[6.7]{BPGR16} a proof of  conjecture 1 for this split case; see also \cite{Mu15}.
Actually these proofs are proofs of conjecture 2, as the polynomials $a^ {\mathbf u}_{\mathbf w,\mathbf v}$ are Laurent polynomials inherited from the description of ${^I\!}\SHH$ as a specialization of the associative Bernstein-Lusztig algebra over $\Z[\SHQ]$: the algebra ${^I\!}\SHH_{\Z[\SHQ]}$  over $\Z[\SHQ]$ defined by these structure constants on the basis $(T_{\mathbf w})_{\mathbf w\in W^+}$ is associative.
 
\medskip
  \par\noindent{\bf 8) Remark.}
  All isomorphisms in \cite{R11} are Weyl-isomorphisms, and, when $G$ is strongly transitive, all isomorphisms constructed in \lc are induced by an element of $G$.

  \subsection{Type $0$ vertices}\label{1.4}

  The elements of $Y$, through the identification $Y=N.0 \,\subset\A$, are called {\it vertices of type $0$} in $\A$; they are special vertices. We note $Y^+=Y\cap\sht$, $Y^ {+g}=Y\cap(\sht^\circ\cup V_{0})$, $Y^ {+0}=Y\cap V_{0}$ and $Y^{++}=Y\cap \overline{C^v_f}$.
   The type $0$ vertices in $\SHI$ are the points on the orbit $\SHI_0$ of $0$ by $G$. This set $\SHI_0$ is often called the affine Grassmannian as it is equal to $G/K$, where $K =$ Stab$_G(\{0\})$. But in general, $G$ is not equal to $KYK=KNK$ \cite[6.10]{GR08} \ie $\SHI_0\not=K.Y$.

   \par We know that $\SHI$ is endowed with a $G-$invariant preorder $\leq $ which induces the known one on $\A$.
    Moreover, if $x\leq y$, then $x$ and $y$ are in a same apartment.
    
   We set $\SHI^+=\{x\in\SHI\mid0\leq x\}$ , $\SHI^+_0=\SHI_0\cap\SHI^+$, $G^ {+}=\{g\in G\mid0\leq g.0\}$ and $G^ {+g}=\{g\in G\mid0\stackrel{o}{\leq} g.0\}$; so $\SHI^+_0=G^+.0=G^+/K$.
   As $\leq $ (\resp $\stackrel{o}{\leq}$) is a $G-$invariant preorder, $G^+$ (\resp $G^ {+g}$) is a semigroup, called the {\it Kac-Moody-Tits semigroup} (\resp the {\it generic Kac-Moody-Tits semigroup}).
   
\par One has $G^+=KNK$; more precisely the map $Y^ {++}\to K\backslash G^+/K$ is a bijection, if we identify $\ql\in Y^ {++}\subset W^v\ltimes Y=W=N/\ker\qn$ with its class in $N$ modulo $\ker\qn\subset K$.
Clearly $G^{+g}=K(Y^ {++}\cap Y^ {+g})K$.

\subsection{Vectorial distance}\label{1.5}

    For $x$ in the Tits cone $\sht$, we denote by $x^{++}$ the unique element in $\overline{C^v_f}$ conjugated by $W^v$ to $x$.

    \par  Let $\SHI\times_\leq \SHI=\{(x,y)\in\SHI\times\SHI\mid x\leq y\}$ be the set of increasing pairs in $\SHI$.
    Such a pair $(x,y)$ is always in a same apartment $g.\A$; so $(g^{-1}).y-(g^{-1}).x\in\sht$ and we define the {\it vectorial distance} $d^v(x,y)\in  \overline{C^v_f}$ by $d^v(x,y)=((g^{-1}).y-(g^{-1}).x)^{++}$.
    It does not depend on the choices we made (by \ref{1.11}.1 below).
    
 \cache{ If $(x,y)\in \SHI\times_\leq \SHI$ and $0≠d^v(x,y)=\ql\in  \overline{C^v_f}$, we say that $[x,y)$ (resp. $[y,x)$) is a segment germ of type $\ql$ (resp. $-\ql$).
 We write $\SHS^0_{\pm\ql}$ the set of segment germs of type $\pm\ql$ in $\SHI$, with origin in $\SHI_0$.
 Clearly $\SHS^0_{\pm\ql}=\SHS^0_{\pm k.\ql}$ for any $k\in\R_{>0}$.
 }

    \par For $(x,y)\in \SHI_0\times_\leq \SHI_0=\{(x,y)\in\SHI_0\times\SHI_0\mid x\leq y\}$, the vectorial distance $d^v(x,y)$ takes values in $Y^{++}$.
     Actually, as $\SHI_0=G.0$, $K$ is the  stabilizer of $0$ and $\SHI^+_0=K.Y^{++}$ (with uniqueness of the element in $Y^{++}$), the map $d^v$ induces a bijection between the set $\SHI_0\times_\leq \SHI_0/G$ of $G-$orbits in $\SHI_0\times_\leq \SHI_0$ and $Y^{++}$.

     \par Further, $d^v$ gives the inverse of the  map $Y^{++}\to K\backslash G^+/K$, as any $g\in G^+$ is in $K.d^v(0,g.0).K$.

 \cache{    \par For $x,y\in\A$, we say that $x\leq _{Q^\vee}\,y$ (resp. $x\leq _{Q^\vee_\R}\,y$) when $y-x\in Q^\vee_+$ (resp. $y-x\in Q^\vee_{\R+}=\sum_{i\in I}\,\R_{\geq 0}.\qa_i^\vee$).
     We get thus  an order on $\A$ (as we supposed $(\qa_i^\vee)_{i\in I}$  free): the $Q^ {\vee}-$order (\resp the $Q^\vee._{\R}-$order).  }

\subsection{Paths and retractions}
\label{suse:Paths}

We consider piecewise linear continuous paths
$\pi:[0,1]\rightarrow \mathbb A$ such that each (existing) tangent vector $\pi'(t)$
belongs to an orbit $W^v.\lambda$ for some $\lambda\in {\overline{C^v_f}}$. Such a path is called a {\it $\lambda-$path}; it is
increasing with respect to the preorder relation $\leq$ on $\mathbb A$.
If $\ql\in \ov{C} ^v_{f}\cap(\sht^\circ\cup V_{0})$, then it is increasing for $\stackrel{o}{\leq}$.

 For any $t\neq 0$ (resp. $t \neq1$), we let
$\pi'_-(t)$ (resp. $\pi'_+(t)$) denote the derivative of $\pi$ at $t$ from the left
(resp. from the right). Further, we define $w_\pm(t)\in W^v$ to be the smallest
element  in its
$(W^v)_\lambda-$class such that $\pi'_\pm(t)=w_\pm(t).\lambda$ (where $(W^v)_\lambda$ is the  stabilizer in $W^v$ of $\lambda$).

\par 
Moreover, we denote by $\pi_-(t)=\pi(t)-[0,1)\pi_-'(t)=[\pi(t),\pi(t-\varepsilon)\,)$ (\resp $\pi_+(t)=\pi(t)+[0,1)\pi_+'(t)=[\pi(t),\pi(t+\varepsilon)\,)$ (for $\varepsilon>0$ small) the negative (\resp positive) segment-germ of $\pi$ at $t$, for $0<t\leq1$ (\resp $0\leq t<1$).

\bigskip

\par Let $C_{z}$ (\resp $\g S$) be a local chamber with vertex $z$ (\resp a sector germ) in an apartment $A$ of $\SHI$.
For all $x\in \SHI_{\geq z}=\{y\in\SHI\mid y\geq z\}$ (\resp $x\in \SHI$) there is an apartment $A'$ containing $x$ and $C_{z}$ (\resp $\g S$).
And this apartment is conjugated to $A$ by an element of $G$ fixing $C_{z}$ (\resp $\g S$) (\cf \ref{1.3}.1.a and \ref{1.3}.4).
 So, by the usual arguments we can define the retraction $\qr=\qr_{A,C_{z}}$ from $\SHI_{\geq z}$ (\resp $\qr=\qr_{A,\g S}$ from 
 $\SHI$) onto the apartment $A$ with center $C_{z}$ (\resp  $\g S$).

 For any such retraction $\qr$, the image of any segment $[x,y]$ with $(x,y)\in\SHI\times_\leq \SHI$ and $d^v(x,y)=\ql\in\overline{C^v_f}$ (with moreover $x,y\in\SHI_{\geq z}$ if $\qr=\qr_{A,C_{z}}$) is a $\ql-$path     \cite[4.4]{GR08}. In particular, $\qr(x)\leq \qr(y)$.
 By definition, if $A'$ is another apartment containing $\g S$ (\resp  $C_{z}$), then $\qr$ induces an isomorphism from $A'$ onto $A$.
    As we assume the existence of the strongly transitive group $G$, this isomorphism is the restriction of an automorphism of $\SHI$.

\subsection{Preordered convexity}\label{1.11}

\cache{In this section we recall the preordered convexity of intersections of apartments and deduce from that result some consequences for chambers, in particular the definitions of the projections of a chamber on a point and on a \red{generic} germ of segment. 

\medskip
\parni {\bf 1) Preordered convexity  property.} }

Let $\mathscr C^{\pm}$ (\resp $\mathscr C_0^{\pm}$) be the set of all local chambers of direction $\pm$ (\resp with moreover vertices of type $0$).
A positive (resp. negative) local chamber of vertex $x\in\SHI$  will often be written  $C_x$ (\resp $C_x^-$) and its direction $C_x^v{=\vect {C_x}}$ (\resp $C_x^{-\,v}{=\vect {C_x^-}}$). 
We consider the set $\mathscr C^+\times_\leq \mathscr C^+=\{(C_x,C_y)\in \mathscr C^+\times \mathscr C^+\mid x\leq y\}$ (\resp $\mathscr C^+\times_{\leq}^\circ\mathscr C^+=\{(C_x,C_y)\in \mathscr C^+\times \mathscr C^+\mid x \stackrel{o}{\leq} y\}$).
 We sometimes write $C_x\leq C_y$ (\resp $C_x\stackrel{o}{\leq} C_y$) when $x\leq y$ (\resp $x\stackrel{o}{\leq}y$).

\begin{prop*}\label{New.1.11}   Let $x,y\in\SHI$ with $x\leq y$.
We consider two local faces $F_x,F_y$ with respective vertices $x,y$. Then

\par (a) $F_x$ and $F_y$ are contained in a same apartment.

\par (b) If $A,B$ are two apartments containing $\{x,y\}$ (\resp $F_x\cup F_y$), then there is a Weyl-isomorphism from $A$ onto $B$, fixing the enclosure $cl_A^\#(\{x,y\})=cl_{B}^\#(\{x,y\})\supset[x,y]$ (\resp the closed convex hull $\ov{conv}_A(F_x\cup F_y)=\ov{conv}_B(F_x\cup F_y)$).
\end{prop*}

\par This improvement of results in \cite[5.4, 5.1]{R11} and \cite[1.10]{BPGR16} is proved by Auguste Hébert: \cite[5.17, 5.18]{He17a}, see also \cite[4.4.16, 4.4.17]{Heb18}.
In b) the case of $\{x,y\}$ is proved in \cite[5.4]{R11} as, by \cite[5.1]{He17a} or \cite[4.4.1]{Heb18}, one may replace $cl$ by $cl^ {\#}$.
This property is called the {\it preordered convexity} of intersections of apartments.

\medskip

\parni{\bf Consequence.} We define $W^+ = W^v\ltimes Y^+$ (\resp $W^ {+g} = W^v\ltimes Y^ {+g}$) which is a subsemigroup of $W$, and call it the {\it Tits-Weyl} (\resp {\it generic Tits-Weyl}) {\it semigroup}.
An element $\mathbf{w}\in W^ {+g}$ is called {\it generic} (in a large sense) and {\it spherical} if, moreover, $\ql\in \sht^\circ\cap Y^+$.

\par {Let $\qe,\eta\in\{+,-\}$.}
If $C_x^{{\qe}}\in\SHC_0^{{\qe}}$ {and $0≤x$}, we know by b) above, that there is an apartment $A$ containing $C_0^{{\eta}}$ and $C_x^{{\qe}}$.
But all apartments containing $C_0^{{\eta}}$ are conjugated to $\A$ by $K_I^{{\eta}}$ (by \ref{1.3}.1.a), so there is $k\in K_I^{{\eta}}$ with $k^{-1}.C_x^{{\qe}}\subset\A$.
 Now the vertex $k^{-1}.x\in\SHI_0$ of $k^{-1}.C_x^{{\qe}}$ satisfies $k^{-1}.x\geq0$, so there is $\mathbf{w}\in W^+$ such that $k^{-1}.C_x^{{\qe}}=\mathbf{w}.C_0^{{\qe}}$.

\par When $g\in G^+$, $g.C_0^{{\qe}}$ is in $\SHC_0^{{\qe}}$ and there are $k\in K_I^{{\eta}}$, $\mathbf{w}\in W^+$ with $g.C_0^{{\qe}}=k.\mathbf{w}.C_0^{{\qe}}$, \ie $g\in K_I^{{\eta}}.W^+.K_I^{{\qe}}$.
We have proved the {\it Bruhat decompositions} $G^+=K_I^{\pm}.W^+.K_I^{\pm}$
and the {\it Birkhoff decompositions} $G^+=K_I^\mp.W^+.K_I^\pm$.
 {For uniqueness, see \ref{1.13} below.}
 
 \par Similarly we have also $G^ {+g}=K_I^{\pm}.W^ {+g}.K_I^{\pm}$
and  $G^ {+g}=K_I^\mp.W^ {+g}.K_I^\pm$.

\begin{rema}\label{1.11c} \par If the generalized Cartan matrix $\M$ is of affine or strictly hyperbolic type (in the sense of \cite[4.3 or Ex. 4.1]{K90}), then any non spherical vectorial face is $w.F^v(I)=F^v(I)=V_{0}=\{ v\in V \mid \qa_{i}(v)=0, \forall i\in I\}$.
So the Tits cones satisfy $\sht=\sht^\circ\sqcup V_{0}$ and $Y^+=Y^ {+g}$, $W^+=W^ {+g}$ .
\end{rema}

\subsection{$W-$distance}\label{1.13}

Let $(C_x,C_y)\in\mathscr C_0^+\times_\leq \mathscr C_0^+$, there is an apartment $A$ containing $C_x$ and $C_y$.
We identify $(\A,C_0^+)$ with $(A,C_x)$ \ie we consider the unique $f\in Isom^W_\R(\A,A)$ such that $f(C_0^+)=C_x$.
 Then $f^{-1}(y)\geq 0$ and there is $\mathbf{w}\in W^+$ such that  $f^{-1}(C_y)=\mathbf{w}.C_0^+$.
 By \ref{1.11}.b, $\mathbf{w}$ does not depend on the choice of $A$.

\par We define the {\it $W-$distance} between the two local chambers $C_x$ and $C_y$ to be this unique element: $d^W(C_x,C_y) = \mathbf{w}\in W^+ = Y^+\rtimes W^v$.
 If $\mathbf{w}=\ql.w$, with $\ql\in Y^+$ and $w\in W^v$, we write also $d^W(C_x,y)=\ql$; it implies $d^v(x,y)=\ql^ {++}$.
 As $\leq$ is $G-$invariant, the $W-$distance is also $G-$invariant.
 When $\mathbf{w}=w\in W^v$ and $w=r_{i_{1}}.\cdots.r_{i_{r}}$ is a reduced decomposition, we have $d^W(C_{x},C_{y})=w$ if and only if there is a minimal gallery (of local chambers in $\sht^+_x\SHI$) from $C_{x}$ to $C_{y}$ of type $(i_{1},\ldots,i_{r})$, in particular $x=y$. 
When $x=y$, this definition coincides with the one in \ref{1.3}.2.

\cache{ \par If $C_x,C_y,C_z\in\SHC_0^+$, with $x\leq y\leq z$, are in a same apartment, we have the Chasles relation: $d^W(C_x,C_z)=d^W(C_x,C_y).d^W(C_y,C_z)$.}

\par Let us consider an apartment $A$ and local chambers $C_x,C_y,C_z\in\SHC_0^+$ included in $A$.
If $d^W(C_{x},C_{y})=\mathbf{w}$, we write $C_{y}=C_{x}*\mathbf{w}$.
Conversely, for any $\mathbf{w}\in W^+$, there is a unique  local chamber $C_{z}=C_{x}*\mathbf{w}$ in $A$ such that $d^ {W}(C_{x},C_{z})=\mathbf{w}$; actually $C_{x}*\mathbf{w}$ depends on $A$, but not on an identification of $A$ with $\A$.
For $x\leq y\leq z$, we have (in $A$) the Chasles relation: $d^W(C_x,C_z)=d^W(C_x,C_y).d^W(C_y,C_z)$; \ie $(C_{x},\mathbf{w}) \mapsto C_{x}*\mathbf{w}$ is a right action of the semi-group $W^+$.
When $(A,C_{x})$ is identified with $(\A,C^+_{0})$, one has $C_{x}*\mathbf{w}=\mathbf{w}C_{x}$.

\par When $C_x=C_0^+$ and $C_y=g.C_0^+$ (with $g\in G^+$), $d^W(C_x,C_y)$ is the only $\mathbf{w}\in W^+$ such that $g\in K_I.\mathbf{w}.K_I$.
This is the uniqueness result in Bruhat decomposition: $G^+=\coprod_{\mathbf{w}\in W^+}\,K_I.\mathbf{w}.K_I$.
Similarly we have $G^ {+g}=\coprod_{\mathbf{w}\in W^ {+g}}\,K_I.\mathbf{w}.K_I$.

The $W-$distance classifies  the orbits of $K_I$ on $\{C_y\in\SHC_0^+\mid y\geq 0\}$, hence also the orbits of $G$ on $\mathscr C_0^+\times_\leq \mathscr C_0^+$.


\subsection{Iwahori-Hecke Algebras}\label{s2}

We consider any commutative ring with unity  $R$.
The {\it Iwahori-Hecke algebra}   $^I\mathcal H_R$  associated to $\SHI$ with coefficients in $R$ introduced in \cite{BPGR16} is as follows:
 
 To each $\mathbf w\in W^+$, we associate a function $T_{\mathbf w}$ from  $\mathscr C_0^+\times_\leq \mathscr C_0^+$ to $R$ defined by  
$$
T_{\mathbf w}(C,C') =
\left\{
\begin{array}{l}
1 \quad\hbox{ if } d^W(C,C') = \mathbf w,\\
0 \quad\hbox{ otherwise.}
\end{array}
\right.
$$
 The Iwahori-Hecke algebra   $^I\mathcal H_R$  is the free $R-$module
$$\{\sum_{\mathbf w\in W^+} a_{\mathbf w}T_{\mathbf w}\mid a_{\mathbf w}\in R,\ a_{\mathbf w} = 0 \hbox{ except for a finite number of } \mathbf w\},
$$  endowed with the convolution product:

$$
(\qf*\psi)(C_x,C_y)=\sum_{C_z}\,\qf(C_x,C_z)\psi(C_z,C_y).
$$ where $C_z\in \mathscr C^+_0$ is such that $x\leq z \leq y$. 

\par Actually, $^I\mathcal H_R$ can be identified with the natural convolution algebra of the functions $G^+\to R$, bi-invariant under $K_I$ and with finite support (in $K_{I}\backslash G^+/K_{I}$); this is the definition given in the introduction.

\par More precisely $^I\mathcal H_R$ is the space of  functions $\qf: \mathscr C_0^+\times_\leq \mathscr C_0^+ \to R$, that are left $G-$invariant and with support a finite union of orbits (see the last alinea of \ref{1.13}).
To a $\qf \in {^I\mathcal H_R}$ is associated $\qf^G: K_{I}\backslash G^+/K_{I} \to R$ such that $\qf^G(g)=\qf(C_{0}^+,g.C_{0}^+)$.
So, for $\qf,\psi \in {^I\mathcal H_R}$, 

\medskip
\par $(\qf*\psi)^G(g)=(\qf*\psi)(C_{0}^+,g.C_{0}^+)=\sum_{C_{z}}\,\qf(C_{0}^+,C_{z})\psi(C_{z},g.C_{0}^+)$

$\qquad\qquad\quad=\sum_{h\in G^+/K_{I}}\,\qf(C_{0}^+,h.C_{0}^+)\psi(h.C_{0}^+,g.C_{0}^+)$

$\qquad\qquad\quad=\sum_{h\in G^+/K_{I}}\,\qf(C_{0}^+,h.C_{0}^+)\psi(C_{0}^+,h^ {-1}g.C_{0}^+)=\sum_{h\in G^+/K_{I}}\,\qf^G(h)\psi^G(h^ {-1}g)$:

\medskip
\parni we get the convolution product (in the classical case, we take a Haar measure on $G$ with $K_{I}$ of measure $1$).

\medskip
\par One considers also the subspace $^I\mathcal H_R^g=\sum_{\mathbf w\in W^ {+g}} R.T_{\mathbf w}$.
From \ref{4.3} and Remark \ref{sc7}.2 one sees that it is a subalgebra of $^I\mathcal H_R$.
We call it the {\it generic Iwahori-Hecke algebra} associated to $\SHI$ with coefficients in $R$.
From \ref{1.11c} one has $^I\mathcal H_R=^I\mathcal H_R^g$ in the affine or strictly hyperbolic cases.

We recall now some useful results of  \cite{BPGR16} in order to introduce the structure constants and a way to compute them.

\begin{prop}\cite[2.3]{BPGR16}\label{PrFinite2}

Let us fix two local chambers $C_x$ and $C_y$ in $\mathscr C_0^+$ with $x\leq y$ and $d^W(C_x,C_y) = \mathbf u\in W^+$. We consider $\mathbf w$ and $\mathbf v$ in $W^+$. Then the number $a_{\mathbf w,\mathbf v}^{\mathbf u}$ of $C_z\in \mathscr C_0^+$ with $x\leq z\leq y$, $d^W(C_x,C_z) = \mathbf w$ and $d^W(C_z,C_y) = \mathbf v$ is finite (i.e. in $\mathbb N$).

\cache{If we assume $\mathbf w = \lambda$, $\mathbf v = \mu$ and $\mathbf u = \nu$, then $a_{\mathbf w,\mathbf v}^{\mathbf u} = a_{\lambda,\mu}^\nu \geq 1$ (resp. $=1$) when $\ql\in Y^{++}$, $\qm\in Y^{+}$ (resp. $\ql,\qm\in Y^{++}$) and $\qn=\ql+\qm$.}
\end{prop}

\begin{theo}\cite[2.4]{BPGR16}\label{ThAlgebra}

For any ring $R$, $^I\mathcal H_R$ is an algebra with identity element $Id = T_1$ such that
$$
T_{\mathbf w} * T_{\mathbf v} = \sum_{\mathbf u\in P_{\mathbf w,\mathbf v}} a_{\mathbf w, \mathbf v}^{\mathbf u} T_{\mathbf u}
$$ where $P_{\mathbf w,\mathbf v}$ is a finite subset of  $W^+$, such that $a_{\mathbf w, \mathbf v}^{\mathbf u}=0$ for $\mathbf u\notin P_{\mathbf w,\mathbf v}$.

\end{theo}

\section{Projections and retractions}\label{pr}

In this section we introduce the new tools that we shall use in the next section to compute the structure constants of the Iwahori-Hecke algebra.

\subsection{Projections of chambers}
\label{sc0}

\medskip
\parni {\bf 1) Projection of a chamber $C_y$ on a point $x$.} 

 Let $x\in \SHI$, $C_y\in\SHC^+$ with $x\leq y$, $x\not=y$.
 We consider an apartment $A$ containing $x$ and $C_y$ (by \ref{1.11} (a) above) and write $C_y=F(y,C^v_y)$ in $A$.
 For $y'\in y+C^v_y$ sufficiently near to $y$, $\qa(y'-x)\not=0$ for any root $\qa$ and $y'-x\in\sht^\circ$.
 So $]x,y')$ is in a unique positive local chamber $pr_x(C_y)$ of vertex $x$; this chamber satisfies $[x,y)\subset \overline{pr_x(C_y)}\subset cl_A(\{x,y'\})$ and does not depend of the choice of $y'$.
 Moreover, if $A'$ is another apartment containing $x$ and $C_y$, we may suppose $y'\in A\cap A'$ and $]x,y')$, $cl_A(\{x,y'\})$, $pr_x(C_y)$ are the same in $A'$.
 The  local chamber $pr_x(C_y)$ is well determined by $x$ and $C_y$, it is the {\it projection} of $C_y$ in
$\sht^+_x\SHI$.

\par The same things may be done changing  $+$ to $-$ or $\leq$ to $\geq$.
 But, in the above situation, if $C_y\in\SHC^-$, we have to assume $x\stackrel{o}{<} y$ to define $pr_x(C_y)\in\SHC^+$: otherwise $]x,y')$ might be outside $x+\sht$.
  
  \medskip
\parni {\bf 2) Projection of a chamber $C_y$ on a generic segment germ}

  Let $x\in \SHI$, $\qd =[x,x')$ a generic segment-germ and  $C_y\in\SHC$ with $x\leq y$.
 By 1) we can consider $pr_x(C_y)\in\SHC^+$ (with the hypothesis $x\stackrel{o}{<} y$  if $C_y\in \SHC^-$).  We consider now an apartment $A$ containing $[x,x')$ and $pr_x(C_y)$ (by a) above).

 We consider  inside $A$  the prism  denoted by $prism_\qd (C_y)$ obtained as the intersection of all half-spaces $D(\qa, k)$ (for $\qa\in \Phi$ and $ k\in \R$)  that contain $pr_x(C_y)$ and such that $\qd\subset  M(\qa, k)$. 
We can see that if  $\qd =[x,x')$ is regular, $prism_\qd (C_y)=A$.
If the apartment $A$ contains $\qd$ and $C_{y}$ (hence also $pr_x(C_y)$) we may replace $pr_x(C_y)$ by $C_{y}$ in the above definition of $prism_\qd (C_y)$.

\begin{lemm}\label{V2.1.11} In $prism_\qd (C_y)$, there is a unique local chamber of vertex $x$ that contains $\qd$ in its closure. This chamber is independant of the choice of $A$. 
\end {lemm}
\NB This local chamber is, by definition, the {\it projection} $pr_\qd (C_y)$ of the chamber $C_y$ on the segment-germ $\qd$. 
It is the local chamber containing $\qd$ in its closure which is the nearest from $pr_x(C_y)$: either $d^W(pr_x(C_y),pr_\qd (C_y))$ is minimum or $d^ {*W}(pr_x(C_y),pr_\qd (C_y))$ is maximum.

\par The same things may be done when one supposes $y\leq x$ and $C_{y}\in \SHC^ {-}$ or $y\stackrel{o}{<} x$ and $C_{y}\in \SHC^ {+}$.

\begin{proof}
In the apartment $A$, we consider $\qd_+$ the segment-germ  $\qd$ if $\qd$  is in  $ \sht^+_x\SHI$ and $op_A(\qd)$  if  $\qd\in \sht^-_x\SHI$ (where $op_A(\qd)$ denote the opposite segment-germ in $A$). By \ref{1.3}.2, we can consider in the building $ \sht^+_x\SHI$ the minimal galleries from $pr_x(C_y)$ to $\qd_{+}$ (more exactly  to a chamber $C$ such that $\qd_{+}\in \bar C$). The last chamber of each of these galleries is the same, we denote it $C_x^{++}$.
 This chamber is associated to a positive system of roots $\Phi^+$ and a root basis $(\qa_1,\ldots, \qa_l)$, satisfying $\qa_i(\qd)=0\iff  i\leq r  $, where $0\leq r<\ell$ (we identify $x$ and $0$). 
Then, we have the characterization of the prism~: $p\in prism_\qd (C_y) \iff (\qa_i(p)\geq 0 \mathrm{\ for\ } 1\leq i\leq r ) $.  
 We consider $w_r$ the  element of highest length  in the finite Weyl group $\langle (r_{\qa_i} )_{i\leq r}\rangle$.  
 
 The local chamber $C_x^{++}$ if $\qd\in \sht^+_x\SHI$ (\resp $op_A (w_r(C_x^{++})$ if not)  is the unique chamber with vertex $x$ of $prism_\qd (C_y)$ that contains $\qd$ in its closure. Indeed, if $C$ is such a chamber, then if $
]x,p)\subset C$, we have $\qa_i(p) >0 $ for all $i\leq r$  (because  $C\subset prism_\qd (C_y)$) and $\qa_i(p) $ of the same sign as $\qa_i(\qd) $ if $i>r$ (because $\qd\subset  \bar C$) . So $C=C_x^{++}$ if $\qd\in \sht^+_x\SHI$ (\resp $C=op_A(w_r(C_x^{++}))$ if $\qd\in \sht^-_x\SHI$). 
 
 In the case $\qd\in \sht^+_x\SHI $, the characterization of $C_x^{++}$ in  the building $ \sht^+_x\SHI$ proves that it does not depend of the choice of $A$.

The chamber $op_A(w_r(C_x^{++}))$ also only depends on  $\qd$ and $C_y$ if $\qd\in \sht^-_x\SHI$. It is sufficient to prove that it intersects $conv_\A(\qd \cup pr_x(C_y))$.
Indeed, let us choose $\qx$ and $y$ such that $[x,\qx)= \qd$  and $]x,y)\subset pr_x(C_y)$.  
We have  $\qa_i(\qx) =0 $ for $i\leq  r$, $\qa_i(\qx) <0 $ for $i> r$ and  $\qa_i(y) >0 $ for $i\leq r$.
 So for $t$ near $1$ enough, $\qa_i(t\qx+(1-t)y) >0 $ for $i\leq r$ and $<0$ for $i>r$, so   $]x, t\qx+(1-t)y)\subset  op_A(w_r(C_x^{++})$.
By Proposition \ref{1.11}, the local chamber $op_A(w_r(C_x^{++}))$ is included in all apartments containing $\qd$ and $pr_x(C_y)$, so is independent of the choice of $A$.
 \end{proof}

\subsection{Centrifugally folded galleries of chambers}
\label{sc1}

Let $z$ be a point in the standard apartment $\mathbb A$. We have twinned buildings $\mathcal T_z^+\SHI$ (resp. $\mathcal T_z^-\SHI$).
   As in \ref{1.3}.2, we consider their unrestricted structure, so the associated Weyl group is $W^v$ and the chambers (resp. closed chambers) are the local chambers $C=germ_z(z+C^v)$ (resp. local closed chambers $\overline{C}=germ_z(z+\overline{C^v})$), where $C^v$ is a vectorial chamber, \cf \cite[4.5]{GR08} or \cite[{\S{}} 5]{R11}.
The distances (resp. codistances) between these chambers are written $d^W$ (resp. $d^{*W}$).
  To $\A$ is associated a twin system of apartments $\A_z = (\A_z^-,\A_z^+)$.
  
\par Let $\mathbf i = (i_1,..., i_r)$ be the type of a minimal gallery.    We choose in $\A^-_z$ a negative (local) chamber $C^-_z$ and denote by $C^+_z$ its opposite in $\A^+_z$.
We consider now galleries of (local) chambers $\mathbf c = (C_z^-,C_1,...,C_r)$ in the
apartment $\mathbb A_z^-$ starting at $C_z^-$ and of type $\mathbf i$.
Their set is written $\Gamma (C^-_{z},\mathbf i)$.
We consider the root  $\beta_j$ corresponding to the common
limit hyperplane $M_j = M({\beta_j},-\qb_j(z))$ of type $i_j$ of $C_{j-1}$ and $C_j $  satisfying moreover $\qb_j(C_j)\geq{}\qb_j(z)$.

  \par  
  We consider the system of positive roots $\QF^+$ associated to $C^+_z$. Actually, $\QF^+=w.\QF^+_f$, if $\QF^+_f$ is the system $\QF^+$ defined in \ref{1.1} and $C^+_z=germ_z(z+w.C^v_f)$.   
   We denote by $(\qa_i)_{i\in I}$ the corresponding basis of $\QF$ and by $(r_i)_{i\in I}$ the corresponding generators of $W^v$. Note that this change of notation for $\QF^+$ and $r_{i}$ is limited to subsection \ref{sc1}.

The set $\Gamma (C^-_{z},\mathbf i)$ of  galleries is in bijection with the set $\Gamma (\mathbf i) = \{1,r_{i_1}\}\times\cdots\times \{1,r_{i_r}\}$ via the map $(c_1,...,c_r)\mapsto (C_z^-, c_1 C_z^-,...,c_1\cdots c_r C_z^-)$.
Moreover $\beta_j = -c_1\cdots c_j (\alpha_{i_j})$. 

\begin{defi*}
Let $\mathfrak Q$ be a chamber in $\mathbb A_z$.
A gallery $\mathbf c = (C_z^-,C_1,...,C_r)\in\Gamma (C^-_{z},\mathbf i)$ is said to be {\it centrifugally folded} with respect to $\mathfrak Q$ if $C_j = C_{j-1}$ implies that  $M_j $ is a wall and separates $\mathfrak Q$ from $C_j = C_{j-1}$. 
We denote this set of centrifugally folded galleries by $\Gamma^+_{\mathfrak Q} (C^-_{z},\mathbf i)$.
We write $\Gamma^+_{\mathfrak Q} (C^-_{z},\mathbf i,C)$ the subset of galleries in $\Gamma_{\mathfrak Q} (C^-_{z},\mathbf i)$ such that $C_{r}$ is a given chamber $C$.
\end{defi*}

\subsection{Liftings of galleries}
\label{sc2}

Next, let $\qr_{\mathfrak Q} : \sht_{z}\SHI \to \mathbb A_z$ be the retraction centered at $\mathfrak Q$. To a gallery of chambers $\mathbf c = (C_z^-,C_1,...,C_r)$ in $\Gamma(C_z^-,\mathbf i)$,
one can associate the set of all galleries of type $\mathbf i$ starting at $C_z^-$ in $\sht_{z}^-\SHI$ that retract onto $\mathbf c$, we denote this set by $\mathcal C_{\mathfrak Q}(C_z^-,\mathbf c)$.
We denote the set of galleries $\mathbf c' = (C_z^-,C'_1,...,C'_r)$ in $\mathcal C_{\mathfrak Q}(C_z^-,\mathbf c)$ that are minimal (i.e. satisfy $C'_{j-1}\ne C'_j$ for any $j$)  by $\mathcal C_{\mathfrak Q}^m(C_z^-,\mathbf c)$. 
Recall from  \cite[Proposition 4.4]{GR13}, that the set $\mathcal C_{\mathfrak Q}^m(C_z^-,\mathbf c)$ is nonempty if, and only if, the gallery $\mathbf c$ is centrifugally folded with respect to $\mathfrak Q$. Recall also from loc. cit., Corollary 4.5, that if $\mathbf c \in \Gamma_{\mathfrak Q}^+(C_z^-,\mathbf i)$, then the number of elements in $\mathcal C^m_{\mathfrak Q}(C_z^-,\mathbf c)$ is:

$$
\sharp \mathcal C^m_{\mathfrak Q}(C_z^-,\mathbf c) = \prod_{j\in J_1} (q_{j} - 1) \times \prod_{j\in J_2} q_{j} 
$$
 where $q_j=q_{M_j}\in\shq$, 
 $$J_1=\{j\in\{1,\cdots,r\}\mid c_j=1\}=\{j\in\{1,\cdots,r\}\mid C_{j-1}=C_{j}\}$$ 
 and $$J_2=\{j\in\{1,\cdots,r\}\mid C_{j-1}\neq C_{j} \mathrm{\ and\ } M_j \mathrm{\ is\ a\ wall\ separating\ } {\mathfrak Q} \mathrm{\ from\ } C_j\}.$$
 
\par One may remark that $\{1,\cdots,r\}$ contains the disjoint union $J_{1}\sqcup J_{2}$, but may be different from it.
The missing $j$ are precisely those $j$ such that $M_{j}$ is not a wall (hence $q_{M_j}$ is not defined).
One has  $\{1,\cdots,r\}=J_{1}\sqcup J_{2}$, when $z$ is a special point, in particular when $z\in Y$.

\bigskip
\par More generally let $\mathbf c^m = (C_z^-,C_1^m,...,C_r^m)$ be the minimal gallery in $\A^-_{z}$ of type $\mathbf i$.
We write $\mathcal C^m(C^ {-}_{z},\mathbf i)$ the set of all minimal galleries in $\SHI$ of type $\mathbf i$ starting from $C_z^-$. 
Its cardinality is $\prod_{j\in J_2} q_{j}$, where $J_2$ is the set of $1\leq j\leq r$ such that the hyperplane $M_{j}$ separating $C^m_{j-1}$ from $C^m_{j}$ is a wall.

\begin{NB} The $q_{j}=q_{M_{j}}$ in the above formulas are in the set $\shq$ of parameters.
More precisely, by \ref{1.3}.6, if $M_{j}=M(\qb_{j},k_{j})$ with $\qb_{j}=w.\qa_{i}$ (for some $w\in W^v$, $i\in I$ and $k_{j}\in\Z$), then one has $q_{j}=q_{i}$ if $k_{j}$ is even and $q_{j}=q'_{i}$ if $k_{j}$ is odd.
\end{NB}

\subsection{Hecke paths}
\label{sc3}

The Hecke paths we consider here are slight modifications of those used in \cite{GR13}. 
They were defined in \cite{BPGR16}, or in \cite{BCGR13} (for the classical case).

Let us fix a local  chamber $C_x\in \mathscr C_0\cap \mathbb A$.

\begin{defi*} A Hecke path of shape $\ql\in Y^{++}$ with respect to $C_x$ in $\A$ is a $\ql-$path in $\mathbb A$  that satisfies the following assumptions.
 For all $p=\pi(t)$,  we ask $x \stackrel{o}{<} p$, so we can consider the  local negative chamber $C^-_p=pr_{p}(C_x)$ by \ref{sc0}.1. 
  Then we assume moreover 
 that for all $t\in [0,1]\setminus\{0,1\}$, there exist finite sequences
$(\xi_0=\pi_-'(t),\xi_1,\dots,\xi_s=\pi_+'(t))$ of vectors in $V$ and
$(\beta_1,\dots,\beta_s)$  of  real roots such that, for all $j=1,\dots,s$:
\begin{itemize}
\item [(i)] $r_{\beta_j}(\xi_{j-1})=\xi_j$,
\item[(ii)] $\beta_j(\xi_{j-1})<0$,
\item[(iii)]   $\beta_j(\pi(t))\in\mathbb Z$, \ie $\qp(t)$ is in a wall of direction $\ker\qb_{j}$,
\item[(iv)]   $\qb_j(C^-_{\pi(t)}) <\qb_j(\pi(t))$.
\end{itemize}

\par One says then that these two sequences are a $(W^v_{\pi(t)}, C^-_{\pi(t}))-$chain from $\pi'_-(t)$ to $\pi'_+(t)$. 
Actually $W^v_{\pi(t)}$ is the subgroup of $W^v$ generated by the $r_{\qb}$ such that $M(\qb,-\qb(\qp(t)))$ is a wall.
\end{defi*}

When $t\in ]0,1[$ is such that $s\not=0$, one has $\pi'_{-}(t)\not=\pi'_{+}(t)$,  the path is  centrifugally folded with respect to $C_x$  at $\pi(t)$. 

\begin{lemm}\label{sc3a} Let $\qp\subset\A$ be a Hecke path with respect to $C_{x}$ as above. Then,

\par (a) For $t$ varying in $[0,1]$ and $p=\qp(t)$, the set of vectorial rays $\R_{+}(x-\qp(t))$ is contained in a finite set of closures of (negative) vectorial chambers.

\par (b) There is only a finite number of pairs $(M,t)$ with a wall $M$ containing a point $p=\qp(t)$ for $t>0$, such that $\qp_{-}(t)$ is not in $M$ and $x$ is not in the same side of $M$ as $\qp_{-}(t)$ (but may be $x\in M$).

\par (c) One writes $p_{0}=\qp(t_{0}), p_{1}=\qp(t_{1}), \ldots, p_{\ell_\qp{}}=\qp(t_{\ell_{\qp}})$ with $0=t_{0}<t_{1}< \cdots <t_{\ell_{\qp}-1} <1=t_{\ell_{\qp}}$ the points $p=\qp(t)$ satisfying to (b) above (or $t=0,t=1$).
Then any point $t$ where the path is (centrifugally) folded with respect to $C_{x}$ at $\qp(t)$ appears in the set $\{t_{k} \mid 1\leq k\leq \ell_{\qp}-1\}$.
\end{lemm}

\begin{proof} a) The $\ql-$path $\qp$ is a union of line segments $[p'_{0},p'_{1}]\cup [p'_{1},p'_{2}]\cup \cdots \cup [p'_{n-1},p'_{n}]$.
By hypothesis on Hecke paths, for each point $p=\qp(t)$, $x-p$ is in the open negative Tits cone $-\sht^\circ$ (in particular only in a finite number of closures of negative vectorial chambers).
Let $p\in [p'_{i},p'_{i+1}]$, then $x-p=x-p'_{i}-(p-p'_{i})$ and $\R_{+}(x-p)\in conv(\R_{+}(x-p'_{i}),-\R_{+}(p-p'_{i}))$ and this convex hull is independent of $p$ and only in a finite number of closures of (negative) vectorial chambers (as $(x-p'_{i})\in -\sht^\circ$ and $(p-p'_{i})\in \R_{+}(p'_{i+1}-p'_{i})\subset\sht$). So (a) is proved.

\par b) There is only a finite number of vectorial walls separating (strictly) a chamber in the set of (a) above and a vector $p'_{i}-p'_{i+1}$.
And, for each such vectorial wall, there is only a finite number of walls with this direction meeting the compact set $\qp([0,1])$.
Moreover such a wall meets a segment $]p'_{i},p'_{i+1}]$ at most once or contains $[p'_{i},p'_{i+1}]$ (hence $\qp_{-}(t)\subset M$ for $\qp(t)\in]p'_{i},p'_{i+1}]$).

\par c) The folding points are among $\{p_{1},\ldots,p_{\ell_{\qp}-1}\}$ by (iv) and (ii) above for $j=1$.
\end{proof}

\subsection{Retractions and liftings of line segments}
\label{sc3b}

\par{\bf1) Local study.}

In tangent buildings, the centrifugally folded galleries are related with retractions of opposite segment germs, by the following lemma  proved in \cite[ Lemma 4.6]{GR13}.

\par We consider a point $z\in \A$ and a negative local chamber $C_{z}^-$ in $\A_{z}^-$.
Let $\xi$ and $\eta$ be two segment germs in $\A_z^+ =\A\cap \sht^+_z\SHI$ . Let $-\eta$ and $-\xi$ opposite respectively $\eta$ and $\xi$ in $\A_z^-$.
 Let $\mathbf i$ be the type of a minimal gallery between $C_z^-$ and $C_{-\xi}$, where $C_{-\xi}$ is the negative (local) chamber containing $-\xi$ such that $d^W(C_z^-, C_{-\xi})$ is of minimal length.
  Let $\mathfrak Q$ be a chamber of $\A_z^+$ containing $\eta$. We suppose $\qx$ and $\eta$ conjugated by $W^v_z$.

\begin{lemm*} The following conditions are equivalent:

\par (i) There exists an opposite $\zeta$ to $\eta$ in $\sht_{z}^-\SHI$ such that $\rho_{\A_z, C^-_z} ( \zeta) = -\xi$.

\par (ii)  There exists a gallery $\mathbf c \in \Gamma_{\mathfrak Q}^+(C_{z}^-,\mathbf i)$ ending in $-\eta$.

\par (iii) There exists a $(W^v_z, C_z^-)-$chain from $\xi$ to $\eta$.

\par  Moreover the possible $\qz$ are in one-to-one correspondence with the disjoint union of the sets  $\mathcal C^m_{\mathfrak Q}(C_{z}^-,\mathbf c)$ for $\mathbf c$ in the set $\Gamma_{\mathfrak Q}^+(C_{z}^-,\mathbf i,-\eta)$ of galleries in $\Gamma_{\mathfrak Q}^+(C_{z}^-,\mathbf i)$ ending in $-\eta$.
\end{lemm*}

\parni{\bf 2) Consequence.} Let $C_{x}$ be a positive local chamber in $\A$ and $z\in\A$ a point such that $x \stackrel{o}{<} z$. 
We consider $C^-_{z}=pr_{z}(C_{x})$.
Then one knows that the restriction of the retraction $\qr=\qr_{\A,C_{x}}$ to the tangent twin building $\sht_{z}\SHI$ is the retraction $\rho_{\A_z, C^-_z}$.
 
\par We consider two points $y,z_{0}$  in $\SHI$ such that $x \stackrel{o}{<} z_{0}\leq y$, with $d^v(z_{0},y)=\ql\in Y^ {++}$.
By \ref{suse:Paths}, the image $\qr([z_{0},y])$ is a $\ql-$path $\qp$ from $\qr(z_{0})$ to $\qr(y)$.
For $z\in [z_{0},y[$, we consider an apartment $A$ containing $[z,y)$ and $C_{x}$, hence also $C_{z}^-$.
We write $p=\qr(z)$.
The restriction $\qr\vert_{A}$ is the restriction to $A$ of an automorphism $\qf$ of $\SHI$ fixing $C_{x}$ (and an isomorphism from $A$ to $\A$); $\qf$ induces an isomorphism $\qf\vert_{\sht_{z}\SHI}$ from $\sht_{z}\SHI$ onto $\sht_{z}\SHI$.
One has $\qr\vert_{\sht_{z}\SHI}=\rho_{\A_p, C^-_p}\circ\qf\vert_{\sht_{z}\SHI}=\qf\vert_{A_{z}}\circ\rho_{A_{z}, C^-_{z}}$.
So one may use the above Lemma, more precisely the implication $(i)\implies(iii)$:
we get a $(W^v_p, C_p^-)-$chain from $\qp'_{-}(t)$ to $\qp'_{+}(t)$ (if $p=\qp(t)$).

\par We have proved that $\qp=\qr([z_{0},y])$ is a Hecke path of shape $\ql$ with respect to $C_{x}$ in $\A$.
This result  is a part of  \cite[Theorem 3.4]{BPGR16}. 
It is also a consequence of the proof of \cite[Th. 3.8]{BCGR13} which deals with the classical case of buildings.

\medskip
\parni{\bf 3) Liftings of Hecke paths.}

\par One considers in $\A$ a positive local chamber $C_{x}$, a Hecke path $\qp$ of shape $\ql\in Y^ {++}$ with respect to $C_{x}$ and the retraction $\qr=\qr_{\A,C_{x}}$.
Given a point $y\in\SHI$ with $\qr(y)=\qp(1)$, we consider the set $S_{C_{x}}(\qp,y)$ of all segment germs $[z,y]$ in $\SHI$ such that $\qr([z,y])=\qp$.
The above Lemma (essentially (ii)) is used in \cite{BPGR16} to compute the cardinality of $S_{C_{x}}(\qp,y)$.

\par We consider the notations of \ref{suse:Paths} and the numbers $t_{k}$ of  Lemma \ref{sc3a}.
Then $p_{k}=\qp(t_{k})$, $\qx_{k}=-\qp_{-}(t_{k})$, $\eta_{k}=\qp_{+}(t_{k})$ and 
$\mathbf i_k$ is the type of a minimal gallery between $C_{p_k}^-$ and $C_{-\xi_k}$, where $C_{-\xi_k}$ is the negative (local) chamber such that  $-\xi_k\subset \overline{C_{-\xi_k}}$ and $d^W(C_{p_k}^-, C_{-\xi_k})$ is of minimal length.
 Let $\mathfrak Q_k$ be a fixed chamber  in $\A^+_{z_k}$ containing $\eta_k$ in its closure
and $\Gamma_{\mathfrak Q_k}^+(C_{p_k}^-,\mathbf i_k, -\eta_k)$ be the set of all the galleries $(C^-_{z_k}, C_1,...,C_r )$  of type $\mathbf i_k$ in $\A^-_{z_k}$, centrifugally folded with respect to $\mathfrak Q_k$ and with $-\eta_k\in \overline{C_r}$.

 \par The following result is Theorem 3.4 in \cite{BPGR16}.
 One uses the notations of \ref{sc1} and \ref{sc2}.
One considers paths $\qp$ more general than Hecke paths.
 The idea is to  lift the path $\pi$ step by step starting from its end by using the above Lemma.
 We shall generalize it in Theorem \ref{sc7} by lifting decorated Hecke paths (see just below).

\begin{theo}\label{sc4} The set $S_{C_x}(\pi, y)$ is non empty if, and only if,  $\pi$ is a Hecke path with respect to $C_x$. Then, we have a bijection
$$
S_{C_x}(\pi, y)\simeq \Big( \prod_{k=1}^{\ell_\pi-1} \coprod_{\mathbf c\in\Gamma_{\mathfrak Q_k}^+({C^-_{p_{k}},}\mathbf i_k,-\eta_k)} \mathcal C^m_{\mathfrak Q_k} ({C^-_{p_{k}},}\mathbf c) \Big). \mathcal C^m(C^ {-}_{y},\mathbf i_{\ell_{\qp}})
$$

\par In particular, the number of elements in this set is a polynomial  in the numbers $q\in\shq$ with coefficients in $\Z$ depending only on $\A$.
\end{theo}

\subsection{Decorated segments and paths}\label{sc6b}

Let us consider $z_0$ and $y$ in $\SHI$ such that $z_0\stackrel{o}{<}y$.

\medskip
\parni{\bf 1) Definition.}  A {\it decorated segment} $\underline{[z_{0},y]}$ is the datum of a segment $[z_{0},y]$ as above and, for any $z\in[z_{0},y[$ (\resp $z\in ]z_{0},y]$) of a positive (\resp negative) chamber $C^ {+}_{z}$ (\resp $C''_{z}$) with vertex $z$ and containing the segment germ $[z,y)$ (\resp $[z,z_{0})$) in its closure.
One asks moreover that $C^ {+}_{z}=pr_{[z,y)}(C)$ (\resp  $C''_{z}=pr_{[z,z_{0})}(C)$) for any local chamber $C=C^ {+}_{z'}$ or $C=C''_{z'}$ as above.
One may remark that, then, $C^ {+}_{z}=pr_{z}(C)$ (\resp  $C''_{z}=pr_{z}(C)$) if $z'\in [z,y]$ (\resp $z'\in[z_{0},z]$).

\medskip
\par Clearly the decorated segment $\underline{[z_{0},y]}$ is entirely determined by the segment $[z_{0},y]$ and any of the local chambers $C^ {+}_{z'}$ or $C''_{z'}$. 
It is entirely contained in any apartment containing  $[z_{0},y]$ and one local chamber $C^ {+}_{z'}$ or $C''_{z'}$ (by  \ref{V2.1.11}).

\par For points $z'_{0}\neq y'$ in $[z_{0},y]$ in the order $z_{0},z'_{0},y',y$ (\ie $z'_0\stackrel{o}{<}y'$) the datum $\underline{[z'_{0},y']}=([z'_{0},y'],(C^ {+}_{z})_{z\in[z'_{0},y'[},(C''_{z})_{z\in ]z'_{0},y']})$ is a decorated segment.

\begin{enonce*}[plain]{2) Lemma} Let $[z_{0},y]$ be a segment as above, $z_{1}\in[z_{0},y]$ and $C_{z_{1}}$ a local chamber with vertex $z_{1}$ contained in a same apartment $A$ as $[z_{0},y]$.
Let us define $C^ {+}_{z}=pr_{[z,y)}(C_{z_{1}})$ and  $C''_{z}=pr_{[z,z_{0})}(C_{z_{1}})$.
Then $\underline{[z_{0},y]}=([z_{0},y],(C^ {+}_{z})_{z\in[z_{0},y[},(C''_{z})_{z\in ]z_{0},y]})$ is a decorated segment.
 Moreover in $A$ all chambers $C^ {+}_{z}$ (\resp $C''_{z}$) are deduced from each-other by a translation.
 \end{enonce*}
 
\NB If $z_{1}$ is $z_{0}$ or $y$ then any   local chamber $C_{z_{1}}$ with vertex $z_{1}$ is contained in a same apartment as $[z_{0},y]$.

 \begin{proof} We have to prove that $C^ {+}_{z}=pr_{[z,y)}(C)$ (\resp  $C''_{z}=pr_{[z,z_{0})}(C)$) for any local chamber $C=C^ {+}_{z'}$ or $C=C''_{z'}$.
 Let us recall that  the chamber $C^ {+}_{z}$ (\resp $C''_{z}$) is  the unique chamber, that contains $\qd=[z,y)$ (\resp $\qd=[z, z_0)$) in its closure, of the prism $prism_{\qd}(C_{z_{1}})$ defined in $A$ as the intersection of all half-spaces $D(\qa, k)$ (for $\qa\in \Phi$ and $ k\in \R$)  that contain $C_{z_{1}}$ and such that  $\qd\subset  M(\qa, k)$. 
 In fact each  prism considered  to define all these chambers in these definitions is  the same prism $prism_{[z_0,y]}(C_{z_{1}})$, as $\qd\subset  M(\qa, k) \iff [z_{0},y] \subset M(\qa,k)$.
 Moreover, as already partially remarked in  \ref{sc0}.2, $prism_{[z_0,y]}(C_{z_{1}})=prism_{[z_0,y]}(C)$ for $C=C^ {+}_{z'}$ or $C=C''_{z'}$.
 Indeed, such a $C$ is in $prism_{[z_0,y]}(C_{z_{1}})$ and any $M(\qa,k)$ containing $[z_{0},y]$ cannot cut $C$, so $prism_{[z_0,y]}(C_{z_{1}})=prism_{[z_0,y]}(C)$.
 
  It is now clear that  $C^ {+}_{z}=pr_{[z,y)}(C)$ (\resp  $C''_{z}=pr_{[z,z_{0})}(C)$) for any local chamber $C=C^ {+}_{z'}$ or $C=C''_{z'}$.
 Moreover the translations of vector in the direction of the line of $A$ containing $\qd$ stabilize the prism and exchange the segment germs.
  So the last assertion of the lemma is clear. 
 \end{proof}

\parni{\bf 3) Definitions.}  A {\it decorated $\ql-$path} $\underline\qp$ is the datum of : 
 
 - a $\ql-$path  $\{\pi(t) \mid 0\leq t\leq 1\}$, 
 
-  a positive  (\resp a negative) local chamber  $C^+_{\pi(t)}$ (\resp $C''_{\pi(t)}$)  of vertex $\pi(t)$ for $0\leq t<1$ (\resp $0<t\leq 1$).

 such that there are numbers $0=t'_{0}<t'_{1}<\cdots t'_{r}=1$ satisfying, for any $1\leq i\leq r$,
 
 - $\{ \qp(t) \mid t'_{i-1}\leq t\leq t'_{i} \}$ is a segment $[\qp(t'_{i-1}),\qp(t'_{i})]$,
 
 - $\underline{[\qp(t'_{i-1}),\qp(t'_{i})]}=([\qp(t'_{i-1}),\qp(t'_{i})],(C^ {+}_{\qp(t)})_{t\in[t'_{i-1},t'_{i}[},(C''_{\qp(t)})_{t\in]t'_{i-1},t'_{i}]})$ is a decorated segment 
 (in particular $\qp(t'_{i-1}) \stackrel{o}{<} \qp(t'_{i}))$, hence $\ql$ is spherical. 
 
 \medskip
 \par A {\it decorated Hecke path} of shape $\ql$ with respect to $C_x$ in $\A$ is a decorated $\ql-$path $\underline\qp$ such that the underlying path $\qp$ is a Hecke path of shape $\ql$ with respect to $C_x$ in $\A$.
 One assumes moreover that the numbers $0<t'_{1}<\cdots <t'_{r}=1$ are equal to the numbers $0<t_1< t_2<\cdots <t_{\ell_\pi}=1$ of Lemma \ref{sc3a} above.

\begin{enonce*}[plain]{4) Proposition} Let $\underline{[z_{0},y]}$ be a decorated segment (with $d^ {v}(z_{0},y)=\ql\in Y^ {++}$ spherical), $C_{x}$ a chamber of vertex $x$ in $\A$ with $x\stackrel{o}{<}z_{0}$ (hence $x\stackrel{o}{<}z$ for any $z\in[z_{0},y]$) and $\qr=\qr_{\A,C_{x}}$ the associated retraction.
We parametrize $[z_{0},y]$ by $z(t)=z_{0}+t(y-z_{0})$ in any apartment containing $[z_{0},y]$.
Then $\qr(\underline{[z_{0},y]})=(\qp=\qr\circ z,(C^ {+}_{\qr z(t)}=\,\qr C^ {+}_{z(t)})_{t\in[0,1[},(C^ {*}_{\qr z(t)}=\,\qr C''_{z(t)})_{t\in]0,1]})$ is a decorated Hecke path of shape $\ql$ with respect to $C_x$ in $\A$.
 \end{enonce*}
 
\begin{NB} Conversely a decorated Hecke path is not always the image by $\qr$ of a decorated segment.
 But the calculations of the number of such liftings (as in Theorem \ref{sc4}) is the main ingredient of our main theorem (\ref{sc7} below) generalizing the Theorem 3.7 in \cite{BPGR16}.
 \end{NB}

 \begin{proof} For any $z\in [z_{0},y[$ (\resp $z\in ]z_{0},y]$), we consider an apartment $A^ {+}_{z}$ (\resp $A''_{z}$) containing $C_{x}$ and $C^ {+}_{z}$ (\resp $C''_{z}$). Then $A^ {+}_{z}\cup A''_{z}$ (or $A^ {+}_{z_{0}}$, $A''_{y}$) contains a neighbourhood of $z$ (or $z_{0}$, $y$) in the segment $[z_{0},y]$.
  By compacity of this segment we get numbers $0=t'_{0}<t'_{1}<\cdots t'_{r}=1$ and apartments $A_{i}$ such that $A_{i}$ contains $C_{x}$, $z([t'_{i-1},t'_{i}])$ and either $C^ {+}_{z(t'_{i-1})}$ or $C''_{z(t'_{i})}$.
   By the projection properties of decorated segments, it contains all other $C^+_{z(t)}$ (\resp $C''_{z(t)}$) for $t\in [t'_{i-1},t'_{i}[$ (\resp $t\in ]t'_{i-1},t'_{i}]$).
  As $\qr$ sends isomorphically $A_{i}$ onto $\A$, we get that $\qr(\underline{[z_{0},y]})$ is a decorated $\ql-$path, with underlying path a Hecke path of shape $\ql$ with respect to $C_x$ in $\A$.
  
\par To get that $\qr(\underline{[z_{0},y]})$ is a decorated Hecke path, we have now to prove that the $t'_{i}$ may be replaced by  the $t_{i}$ associated to this Hecke path by Lemma \ref{sc3a}.
 We may apply the following Lemma to $[\qp(t_{i-1}),\qp(t_{i})]$.
  Any apartment $A$ containing $C_{x}$ and $C''_{z(t_{i})}$ contains $[z(t_{i-1}),z(t_{i})]$, hence also $C''_{z(t)}$ for $t_{i-1}< t \leq t_{i}$ and $C^+_{z(t)}$ for $t_{i-1}\leq t <t_{i}$, by the projection properties of decorated segments.
  But $\qr$ induces an isomorphism from $A$ onto $\A$.
  So $([\qp(t_{i-1}),\qp(t_{i})],(\qr C^+_{z(t)})_{t_{i-1}\leq t <t_{i}}, (\qr C''_{z(t)})_{t_{i-1}< t \leq t_{i}})$ is a decorated segment, as expected.
 \end{proof}

 \begin{enonce*}[plain]{5) Lemma} In an apartment $\A$ of a masure $\SHI$, we consider a local chamber $C_{x}$ and a line segment $[p_{0},p_{1}]$ with $x \,\stackrel{o}{<}  p_{0} \leq p_{1}$.
  We suppose that, for any $p\in]p_{0},p_{1}[$ and any wall $M$ containing $p$, then  $[p,p_{0}]$  is in the half-apartment containing $C_{x}$ limited by $M$.
  We consider the retraction $\qr=\qr_{\A,C_{x}}$. Then,
  
  \par for any segment germ $[z_{1},z)$ in $\SHI$ such that $\qr([z_{1},z))=[p_{1},p_{0})$ (hence $\qr(z_{1})=p_{1}$), there is a unique line segment $[z_{1},z_{0}]$ such that $[z_{1},z_{0})=[z_{1},z)$ and $\qr([z_{1},z_{0}])=[p_{1},p_{0}]$.
  More precisely any apartment $A$ containing $C_{x}$ and $[z_{1},z)$ contains $[z_{1},z_{0}]$.
\end{enonce*}

 \begin{proof} Let $A$ be an apartment containing $C_{x}$ and $[z_{1},z)$.
 Up to the isomorphism $\qr$ from $A$ onto $\A$, one may suppose $A=\A$.
 Then $z_{1}=p_{1}$ and $[p_{1},p_{0}]$ satisfies $[p_{1},p_{0})=[p_{1},z)$, $\qr([p_{1},p_{0}])=[p_{1},p_{0}]$ as expected for $[p_{1},z_{0}]$.
 Let us consider another solution $[p_{1},z_{0}]$, so $[p_{1},z_{0})=[p_{1},p_{0})$ and $\qr([p_{1},z_{0}])=[p_{1},p_{0}]$.
 Let $z'$ be the point satisfying $[p_{1},z'] \subset [p_{1},p_{0}] \cap [p_{1},z_{0}]$ that is the nearest from $p_{0}$.
 One has $z'\neq p_{1}$ and one wants to prove that $z'=p_{0}$.
 If $z'\neq p_{0}$, one may consider a minimal gallery $\mathbf c'$  in $\sht_{z'}^-\SHI$  from $C^-_{z'}=pr_{z'}(C_{x})$ to the segment germ $[z',z_{0})$.
 Clearly $\mathbf c=\qr(\mathbf c')$ is a minimal gallery in $\A^-_{z'}$ from $C^-_{z'}$ to the segment germ $[z',p_{0})$.
 If we write $\g Q=C^-_{z'}$, we have $\mathbf c' \in \shc^m_{\g Q}(C^-_{z'},\mathbf c)$, with the notations of \ref{sc2}.
 But by the hypotheses, no wall $M$ containing $z'$ separates strictly $C_{x}$ (\ie $C^-_{z'}$) from $[z',p_{0})$.
 Hence the formula in \ref{sc2} tells that $\shc^m_{\g Q}(C^-_{z'},\mathbf c)$ is reduced to one element : we have $\mathbf c'=\mathbf c$, $[z',z_{0})=[z',p_{0})$, contrary to the hypothesis on $z'$.
 \end{proof}
 
  \begin{enonce*}[plain]{6) Remark} The definitions and results in 3), 4), 5) above are also true if we replace $C_{x}$ by a negative sector germ $\g S$ in $\A$ and $\qr$ by $\qr_{\A,\g S}$.
  The corresponding results of the Lemma are more or less implicit in \cite{BPGR16}, see the last paragraph of proof of Lemma 2.1 or of Proposition 2.3 in \lc
  \end{enonce*}

\section{Structure constants in spherical cases}\label{sc}

In this section, we compute the structure constants $a_{\mathbf w, \mathbf v}^{\mathbf u}$ of the Iwahori-Hecke algebra $^I\mathcal H_R^\SHI$, assuming that $\mathbf v =\mu . v$ and $\mathbf w = \lambda . w$ are spherical, \ie $\mu$  and $\lambda$ are spherical (see 1.1 for the definitions). As in \cite{BPGR16}, we will adapt some results obtained in the spherical case in \cite{GR13} to our situation.

These structure constants depend on the shape of the standard apartment $\A$ and on the numbers $q_M$ of \ref{1.3}.6.
 Recall that the number of (possibly) different parameters is at most $2.\vert I\vert$. We denoted by $\shq=\{q_1,\cdots,q_l,q'_1=q_{l+1},\cdots,q'_l=q_{2l}\}$ this set of parameters.

\par For $ \ql\in Y^+$ spherical, we denote $w_\ql$ (\resp $w^+_\lambda$) the smallest (\resp longest) element $w\in W^v$ such that $w.\lambda\in \overline{C_f^v}$.
We start by several lemmas.

\begin{lemm}\label{sc6} \cite[3.6]{BPGR16} Let $C_x, C_z \in \mathscr C_0^+$ with $x\leq z$ and $\ql \in Y^+$ spherical, $w\in W^v$.
We write $C^-_z=pr_z(C_x)$. Then
$$
d^W(C_x,C_z) = \lambda . w \Longleftrightarrow
\left \{
\begin{array}{l}
d^W(C_x,z) = \lambda \\
d^{*W}(C^-_z,C_z) = w_\lambda^+ w .
\end{array}
\right.
$$
\end{lemm}

\begin{lemm}\label{sc6aa} Let $C_z, C_y \in \mathscr C_0^+$ with $z\,\stackrel{o}{<}\,y$ and $\qm \in Y^+$ spherical, $v\in W^v$.
We write  $C^+_z=pr_z(C_y)$ and $C''_y=pr_{[y,z)}(C^+_z)=pr_y(C^+_z)$. Then

$$ (1) \qquad\qquad
d^W(C_z,C_y) = \mu v \Longleftrightarrow
\left \{
\begin{array}{l}
d^W(C_z,C^+_z) = v(w_{v^{-1}.\mu})^{-1} \\
d^{W}(C^+_z,C_y) = \mu^{++}w_{v^{-1}.\mu}.
\end{array}
\right.\qquad\qquad\qquad\qquad\qquad
$$

   $$(2) \quad d^W(C^+_z,C_y)= \mu^{++}w_{v^{-1}.\mu} \iff d^W(C^+_z,y)=\mu^{++} \mathrm{\ and\ }
d^{*W}(C''_y,C_y)=w^+_{\mu^{++}}w_{v^{-1}.\mu}$$

\end{lemm} 

\begin{proof}

(1) Let us fix an apartment $A'$ containing $C_z$, $C_y$ and so $C^+_z$ and identify $(A',C_z)$ with $(\mathbb A, C^+_0)$.

Let us  suppose that $d^W(C_z,C_y) = \mu v$ and denote $C^+_y:=C^+_z+\mu$. 
Clearly $d^W(C_z, C_z+\mu)=\mu $ and, by Chasles in $A'$,  $\mu .v= d^W(C_z, C_y)=d^W(C_z, C_z+\mu)d^W(C_z+\mu, C_y)$, hence $d^W(C_z+\mu, C_y)=v$ \ie $C_y=(C_z+\mu)*v$ (\cf \ref{1.13}).
By $G-$invariance of $d^W$ and Chasles, we have  $d^W(C_z, C^+_z)=d^W(C_z+\mu, C^+_y)=d^W(C_z+\mu, C_y)d^W(C_y, C^+_y)=vd^W(C_y, C^+_y)$. 
Among the walls containing $[z,y]$, no one separates $C^+_y $ from $C_y$, so the local chamber $C^+_y$  is the closest chamber to $C_y$ among those containing the segment-germ $]y, y+\mu)$ in their closure, \ie $C^+_y =pr_{[y, y+\mu)}(C_y)$ and $d^W(C_y, C^+_y)=w'$ where $w' $ is the smallest $w\in W$ such that 
$]y, y+\mu)\subset  \overline{C_y*w}=\overline{C_{z+\mu}*vw}=\overline{C_{z}*\mu vw}=\mu vw\overline{C_{z}}$,  as we identified $C_z $ with $C^+_0$ . 
As $\mu=y-z$, we can see  $w' $ as the smallest $w\in W$ such that $]z, z+\mu)\subset  vw \overline{C_z}$
  \ie $v^{-1}\mu\in w\overline{C_f^v}$ (as we identified $C_z $ with $C^+_0$), so $w'=(w_{ v^{-1}.\mu} )^{-1}$. 
  Finally, we get $d^W(C_z, C^+_z)=v(w_{ v^{-1}.\mu} )^{-1}$ and so $$d^{W}(C^+_z,C_y) = (d^W(C_z,C^+_z))^{-1} d^W(C_z,C_y)=  w_{ v^{-1}.\mu}v^{-1}\mu v (w_{ v^{-1}.\mu})^{-1}w_{ v^{-1}.\mu} =\mu^{++}w_{ v^{-1}.\mu}.$$ 

In the same way, if we suppose that  $d^W(C_z,C^+_z) = v(w_{v^{-1}.\mu})^{-1}$ and $
d^{W}(C^+_z,C_y) = \mu^{++}w_{v^{-1}.\mu}$, by Chasles we obtain $d^W(C_z,C_y) = \mu v$.

\smallskip
(2)  We consider now the opposite local  chamber at $y$ of $C^+_y$ (\resp $C_y$) in $A'$ which is denoted by  $-C^+_y$ (\resp $-C_y$). If $d^W(C^+_z,C_y)= \mu^{++}w_{v^{-1}.\mu}$, we have $d^W(C^+_z,y)=\mu^{++}=d^W(C^+_z,C^+_y)$ and  $d^W(C^+_y,C_y)=w_{ v^{-1}.\mu}$, so $d^{*W}(-C^+_y,C_y)=w_{ v^{-1}.\mu}$. By the proof of \ref{V2.1.11}, we see that $C''_y$ and $-C^+_y$ are such that $d^W(-C^+_y,C''_y)=d^W(C''_y,-C^+_y)=w^+_{\mu^{++}}$ (the longest element of $W^v_{\mu^{++}}$ the fixer of $\mu^{++}$ in $W^v$). By Chasles in $A'$, we have  $$d^{*W}(C''_y,C_y)=d^W(C''_y,-C_y)=d^W(C''_y,-C^+_y)d^W(-C^+_y,-C_y)=w^+_{\mu^{++}}.w_{ v^{-1}.\mu}.$$

The converse result is clear by Chasles. 
\end{proof}

\subsection{Local study}\label{sc6c}

\par We shall need a partial generalization of Lemma \ref{sc3b}.1 dealing with decorations.

\par  We consider a point $z\in \A$, a negative local chamber $C_{z}^-$ in $\A_{z}^-$ and the retraction $\qr=\qr_{\A_{z},C^-_{z}}$ in $\sht_{z}\SHI$.
Let $C^+_{z}$ (\resp $C^*_{z}$) be a positive (\resp negative) local chamber in $\A_{z}$, we also  introduce  the retraction $\qr'=\qr_{\A_{z},C^+_{z}}$  in $\sht_{z}\SHI$.
Let $\xi$ and $\eta$ be two segment germs in $\A_z^+ =\A\cap \sht^+_z\SHI$ of the same ``type'' (\ie $\eta=[z,z+w.\ql)$, $\qx=[z,z+w'.\ql)$ for some $\ql\in Y^ {++}$ and $w,w'\in W^v$).
We suppose that $\ov{C^+_{z}}$ contains $\eta$ and $\ov{C^*_{z}}$ contains the opposite $-\qx=[z,z-w'\ql)$ of $\qx$ in $\A_{z}$. We denote $-\eta=[z, z-w.\ql) $ the opposite of $\eta $ in $\A_{z}$ and $\widetilde C_{z}=pr_{-\eta}(C^+_{z})$. 
 Let $\mathbf i$ be the type of a minimal gallery from $C_z^-$ to $C_{z}^*$.

\begin{lemm*} The following conditions are equivalent:

\par (i) There exists a segment germ $\zeta$  opposite $\eta$ in $\sht_{z}^-\SHI$ and a negative local chamber $C''_{z}$ containing $\qz$ in its closure such that $\qr ( \zeta) = -\xi$, $\qr (C''_{z})=C^*_{z}$ and $C''_{z}=pr_{\zeta}(C^+_{z})$.

\par (ii)  There exists a gallery $\mathbf c \in \Gamma_{C^+_{z}}^+(C_{z}^-,\mathbf i)$ ending in the local chamber $\widetilde C_{z}$.

\smallskip 

\par  Moreover the possible $(\qz,C''_{z})$ are in one-to-one correspondence with the disjoint union of the sets  $\mathcal C^m_{C^+_{z}}(C_{z}^-,\mathbf c)$ for $\mathbf c$ in the set $\Gamma_{C^+_{z}}^+(C_{z}^-,\mathbf i, \widetilde C_{z})$ .
\end{lemm*}  

\begin{proof}

If $\zeta$, a segment germ  opposite $\eta$ in $\sht_{z}^-\SHI$, and  $C''_{z}$, a  negative local chamber containing $\qz$ in its closure, are  such that $\qr ( \zeta) = -\xi$, $\qr (C''_{z})=C^*_{z}$ and $C''_{z}=pr_{\zeta}(C^+_{z})$, there is a unique minimal gallery $\mathbf {c '}$ from $C_{z}^-$ to $C''_{z}$ of type  $\mathbf i$ (as $\qr$ induces a bijection between the minimal galleries from $C^-_{z}$ to $C''_{z}$ and the minimal galleries from $C^-_{z}$ to $C^*_{z}$). 
The gallery $\mathbf {c }=\qr'(\mathbf {c '}) $ is in $  \Gamma_{C^+_{z}}^+(C_{z}^-,\mathbf i, \widetilde C_{z})$. Indeed, $\qz$ is opposite $\eta$ so $\qr'(\qz)=-\eta$, hence the image of $C''_{z}=pr_{\zeta}(C^+_{z})$ by $\qr'$ is $\widetilde C_{z}=pr_{-\eta}(C^+_{z})$.

Reciprocally, let ${\mathbf c} \in \Gamma_{C^+_{z}}^+(C_{z}^-,\mathbf i)$ be a gallery ending in the local chamber $\widetilde C_{z}$. We can lift this gallery with respect to $\qr'$  while preserving the first chamber $C_{z}^-$ to obtain a minimal gallery $\mathbf {c '}$ of type $\mathbf i$.
 Let us call $C''_{z}$ the last chamber of the lifted gallery. The isomorphism associated to $\qr' $ (see  \ref{suse:Paths}) between an apartement $A_z$ containing $C^+_{z}$ and $C''_{z}$ and $\A_z$ enables us to say that the lifting of $-\eta$ is $\zeta$ a segment germ $\zeta$  opposite $\eta$ in  $A_z$ and $C''_{z}=pr_{\zeta}(C^+_{z})$. As the gallery $\mathbf c$ is of type $\mathbf i$, $\qr $ send  $C''_{z}$ onto the end of the minimal gallery of same type beginning at $C_{z}^-$, so $\qr (C''_{z})=C^*_{z}$. Moreover, $\zeta$ is of the same type that $-\eta$ (and $-\xi$), so $\qr ( \zeta) = -\xi$. 

 From the first paragraph above, we get an injective map $(\qz, C''_{z}) \mapsto \mathbf {c '}$ from the set of pairs $(\qz, C''_{z})$ as in $(i)$ and the disjoint union of the sets  $\mathcal C^m_{C^+_{z}}(C_{z}^-,\mathbf c)$ for $\mathbf c$ in the set $\Gamma_{C^+_{z}}^+(C_{z}^-,\mathbf i, \widetilde C_{z})$: indeed, $\qz$ is fully determined by $C''_{z}$ (and $\ql$).
The second paragraph proves that this map is surjective.
\end{proof}

\subsection{Opposite line segments}\label{sc6d}

\par The following lemma  will be usefull in  Theorem \ref{sc7}.

\begin{lemm*} Let us consider in a masure $\SHI$ two preordered line segments or rays $\qd_{1},\qd_{2}$ in apartments $A_{1},A_{2}$, sharing the same origin $x$.
One supposes the segments germs $germ_{x}(\qd_{1})$ and $germ_{x}(\qd_{2})$ opposite (in any apartment containing them both).
Then there is a line in an apartment $A$ of $\SHI$ containing $\qd_{1}$ and $\qd_{2}$.
In particular, if $\qd_{1},\qd_{2}$ are line segments (\resp rays), then $\qd_{1}\cup\qd_{2}$ is also a line segment (\resp a line).
\end{lemm*}

\begin{proof} The case of line segments is Lemma 4.9 in \cite{GR13}. The case of rays may be deduced from the fact stated in the part 2 of the proof of \cite[Prop. 5.4]{R11}. As we shall not use it, we omit the details.
\end{proof}

\subsection{The main formula}\label{sc7}

Let us fix two local chambers $C_x$ and $C_y$ in $\mathscr C_0^+$ with $x\leq y$ and $d^W(C_x,C_y) = \mathbf u=\qn.u\in W^+$. We consider $\mathbf w=\ql.w$ and $\mathbf v=\qm.v$ in $W^+$. 
Then we know that the structure constant $a_{\mathbf w,\mathbf v}^{\mathbf u}$ is the number of $C_{z_{0}}\in \mathscr C_0^+$ with $x\leq {z_{0}}\leq y$, $d^W(C_x,C_{z_{0}}) = \mathbf w$ and $d^W(C_{z_{0}},C_y) = \mathbf v$; moreover this number is finite, see Proposition \ref{PrFinite2}.
 In Lemmas \ref{sc6} and \ref{sc6aa} we gave conditions equivalent to these $W-$distance conditions.

We  choose the standard apartment $\A$ containing $C_x$ and $C_y$, and we identify $C_{x}$ with the fundamental local chamber $C_{0}^+$. 

\par The datum of $z_{0}$ is equivalent to the datum of the segment $[z_{0},y]$ or of the decorated segment $\un{[z_{0},y]}$ associated, as in \ref{sc6b}.2, to $[z_{0},y]$ and $C_{y}$.
We consider then  the decorated Hecke path $\underline\qp$ image of $\un{[z_{0},y]}$ by the retraction $\qr_{\A,C_{x}}$.

\par To the Hecke path $\qp$ underlying a decorated Hecke path $\un\qp$ are associated $\ell_{\qp}\in\N$ and numbers $t_{0}=0<t_{1}<t_{2}<\cdots<t_{\ell_{\qp}}=1$ as in  Lemma \ref{sc3a} and  Definition \ref{sc6b}.3.
 We write $p_{k}=\qp(t_{k})$. 
 We write  $C^+_{p}$ (\resp $C^*_{p}$ instead of $C''_{p}$) the decorations of $\un\qp$ at a point $p$ of $\qp$.
 We write $C^+_{z}$ (\resp $C''_{z}$) the decorations of a decorated segment at one of its points $z$.
 
 \par We use freely the notations from \ref{sc0}, \ref{sc1} and \ref{sc2}.

\begin{theo*} Assume $\qm$ and $\ql$ spherical.
Then the structure constant $a_{\mathbf w,\mathbf v}^{\mathbf u}$ is given by:

 $$a_{\mathbf w,\mathbf v}^{\mathbf u} = \sum_{\un\qp} \ \prod_{k=0}^ {\ell_{\qp}}\,a_{\un\qp}(k)$$
 
 where $\un\qp$ runs over the decorated Hecke paths in $\A$ of shape $\qm^ {++}$ with respect to $C_{x}$ from $p_{0}=x+\ql=\ql$ to $y=x+\qn=\qn$, and the integers $a_{\un\qp}(k)$ are given by :

 \medskip
 \par (1) $a_{\un\qp}(\ell_{\qp})=\sum_{\mathbf d\in\Gamma_{C_y}^+(C^-_{y},\mathbf i_\ell,\tilde C_y)} \sharp \mathcal C^m_{C_y} (C^-_{y},\mathbf d)$, where $\mathbf i_\ell$ is the type of a fixed minimal gallery from $C^-_{y}$ to $C^*_{y}$  and $\tilde C_y$ is the unique local chamber at $y$ in $\A$ such that 
 $d^{*W}(\tilde C_y,C_y)=w^+_{\mu^{++}}w_{v^{-1}.\mu}$ .
 
 \medskip
 \par (2) For $1\leq k \leq \ell_{\qp}-1$,  $a_{\un\qp}(k)= \sum _{\mathbf c\in\Gamma_{C^+_{p_{k}}}^+(C^-_{p_{k}},\mathbf i_k,\tilde C_{p_{k}})} \sharp \mathcal C^m_{C^+_{p_{k}}} (C^-_{p_{k}},\mathbf c)$, where $\mathbf i_k$ is the type of a fixed minimal gallery from $C^-_{p_{k}}$ to $C^*_{p_{k}}$  and $\tilde C_{p_{k}}=pr_{-\eta_{k}}(C^+_{p_{k}})$ with $-\eta_{k}$ the segment germ of origin $p_{k}$ in $\A$ opposite $\eta_{k}=\qp_{+}(t_{k})$.
  
 \medskip
 \par (3) $a_{\un\qp}(0)=\sum _{\mathbf e\in\Gamma_{C^-_{p_0}}^+(C^+_{p_{0}},\mathbf i,C'_{p_0})} \sharp \mathcal C^m_{C^-_{p_0}} (C^+_{p_{0}},\mathbf e)$, where $\mathbf i$ is the type of a fixed reduced decomposition of $w_{v^{-1}.\mu}.v^{-1}$ and $C'_{p_{0}}$ is the unique local chamber at $p_{0}=\qp(0)$ in $\A$ such that 
 $d^{*W}( C^-_{p_{0}},C'_{p_{0}})= w_\lambda^+ w$.
\end{theo*} 

\begin{remas*} 1) Actually $\prod_{k=1}^ {\ell_{\qp}-1}\,a_{\un\qp}(k)$ is the number of decorated segments $\un{[z_{0},y]}$ such that $\qr(\un{[z_{0},y]})=\un\qp$ and $C^*_{y}=C''_{y}$. It may be zero.

 2) If $a_{\mathbf w,\mathbf v}^{\mathbf u}\neq 0$, then necessarily $\qn$ is spherical (\ie $\mathbf u \in W^ {+g}$), as then any Hecke path of shape $\qm^ {++}$ is increasing for $\stackrel{o}{<}$ (see \ref{suse:Paths}).
The arguments of \cite{BPGR16} are sufficient for this result.

3) From this theorem we deduce that $a_{\mathbf w,\mathbf v}^{\mathbf u}\neq 0$ is equivalent to the following:

- there exists a Hecke path in $\A$ of shape $\qm^ {++}$ with respect to $C_{x}$ from $p_{0}=x+\ql=\ql$ to $y=x+\qn=\qn$,

- there exists a decoration $\un \qp$ of $\qp$ (always true),

- for this decorated Hecke path each of the sets $\Gamma_{C_y}^+(C^-_{y},\mathbf i_\ell,\tilde C_y)$, $\Gamma_{C^-_{p_0}}^+(C^+_{p_{0}},\mathbf i,C'_{p_0})$ and $\Gamma_{C^-_{p_0}}^+(C^+_{p_{0}},\mathbf i,C'_{p_0})$ is non empty.

4) The number of decorated Hecke paths $\un\qp$ as above is finite: we know that the number of paths $\qp$ is finite (it is a consequence of Theorem 3.5 in \cite{BPGR16}) and, as $\qm$ is spherical, the number of decorations of $\qp$ is finite.

\end{remas*}

 \begin{proof} $a_{\mathbf w,\mathbf v}^{\mathbf u}$ is the number of local chambers $C_{z_{0}}\in \mathscr C_0^+$ with $x\leq z_{0}\leq y$, $d^W(C_x,C_{z_{0}}) = \mathbf w$ and $d^W(C_{z_{0}},C_y) = \mathbf v$ (we chose $C_{x},C_{y}$ in $\A$ such that $d^W(C_x,C_y) = \mathbf u$).
We know that this number is finite, see Proposition \ref{PrFinite2}.
The datum of $z_{0}$ is equivalent to the datum of the segment $[z_{0},y]$ or of the decorated segment $\un{[z_{0},y]}$ associated, as in \ref{sc6b}.2, to $[z_{0},y]$ and $C_{y}$.
We use now the retraction $\qr=\qr_{\A,C_{x}}: \SHI_{\geq x}\to\A$.
We have $y=\qr(y)=x+\qn$ and the condition $d^W(C_{x},z)=\ql$ is equivalent to $\qr(z)=x+\ql=p_{0}$.
So $\qr(\un{[z_{0},y]})$ has to be a decorated Hecke path $\un\qp$ as asked in the theorem. 
And we get the formula:

$$
a_{\mathbf w,\mathbf v}^{\mathbf u} = \sum_{\un\pi}\Bigg ( \hbox{number of liftings of }\un\pi\Bigg )\times \Bigg (\hbox{number of } C_{z_{0}} \hbox{  for ${z_{0}}$ given}\Bigg ),$$

 It is possible to calculate like that for
$\qr(C^+_{z_{0}})=C^+_{p_{0}}$ is well determined by the decorated path $\un\qp$. 
 Hence, the number of $C_{z_{0}}$ only depends on $\un\pi$ and not on the lifting of $\un\pi$.
 In \cite[Theorem 3.7]{BPGR16} we argued the same way, but with Hecke paths (without decoration) so we had to suppose $\mu^{++}$ regular to get that $\rho(C^+_z)$ was well determined by the path $\pi$.

\par For short, we write $\ell=\ell_{\qp}$.
We compute the number of liftings of $\un\qp$ by looking successively at the number of liftings of $\un{[p_{\ell-1},p_{\ell}]}$, $\un{[p_{\ell-2},p_{\ell-1}]}$, \ldots, $\un{[p_{0},p_{1}]}$.

\medskip
\par 1) The number $a_{\un\qp}(\ell)$ of liftings of $\un{[p_{\ell-1},p_{\ell}=y]}$ is the number of liftings $[z_{\ell-1},z_{\ell}=y]$ of $[p_{\ell-1},p_{\ell}=y]$ and $C''_{y}$ of $C^*_{y}$ such that $[y,z_{\ell-1})\subset \ov{C''_{y}}$ and $d^{*W}(C''_y,C_y)=w^+_{\mu^{++}}w_{v^{-1}.\mu}$ (by Lemma \ref{sc6aa}.2).
But $[y,z_{\ell-1}]$ is determined by $[y,z_{\ell-1})$ (\cf Lemma \ref{sc6b}.5) and $[y,z_{\ell-1})$ is determined by $c''_{y}$ and $\qm^ {++}$. 
So we just have to count the liftings $C''_{y}$ of $C^*_{y}$.
 By the same way as in the proof of the lemma \ref{sc6c}, we are going to prove that the possible $C''_{y}$ are in one-to-one correspondance with the disjoint union of the sets $\mathcal C^m_{C_y}(C_{y}^-,\mathbf c)$ for $\mathbf c$ in  $\Gamma_{C_y}^+(C_{y}^-,\mathbf i_\ell,\tilde C_y)$.
In this case, the tools are $\qr=\qr_{\A,C_{x}}$,  that on $\sht_{y}\SHI$, coincides with $\qr=\qr_{\A,C_{y}^-}$(\ref{sc3b}.2) and  $\qr'=\qr_{\A,C_{y}}$. 

If $C''_{y}$ is given, there is a unique minimal gallery $\mathbf c'$ from $C_{y}^-$ to $C''_{y}$ of type $\mathbf i_\ell$ (as $\qr$ induces a bijection between the minimal galleries from $C_{y}^-$ to $C''_{y}=pr_{[y,z_{\ell-1})}(C_{y}) $ and those from $C_{y}^-$ to $C^*_{y}=pr_{[y,p_{\ell-1})}(C_{y}) $). By Lemma \ref{sc6aa}(2) we know that $d^{*W}(C''_y,C_{y})=w^+_{\mu^{++}}w_{v^{-1}.\mu}$, so $\qr' (C''_y)=\tilde C_{y}$, and the gallery $\mathbf c=\qr'(\mathbf c')$ is in $\Gamma_{C_{y}}^+(C_{y}^-,\mathbf i_\ell,\tilde C_{y})$, while $\mathbf c'$ is in $\mathcal C^m_{C_{y}}(C_{y}^-,\mathbf c)$.

Reciprocally, if $\mathbf c$  is in the set $\Gamma_{C_{y}}^+(C_{y}^-,\mathbf i_\ell,\tilde C_{y})$, let us consider $C''_{y}$ the last chamber of $\mathbf c'$ a lifted gallery of $\mathbf c$ with respect to $\qr'$. The condition on $\tilde C_{y}$ enables to say that $d^{*W}(C''_y,C_{y})=w^+_{\mu^{++}}w_{v^{-1}.\mu}$ and so, by lemma \ref{sc6aa} the decoration $C''_{y}$ of $\un{[z_{\ell-1}, y]}$ at $y$ satisfies the expected codistance condition. 

\par 2) For $1\leq k\leq \ell-1$, we suppose given the lifting $\un {[z_k, y]}$ of $\un {\pi}\vert_{[t_k,1]}$. The number $a_{\un\qp}(k)$   of suitable liftings  $\un{[z_{k-1}, z_k]}$ of $\un{[p_{k-1}, p_k]}$ is the number of pairs $([z_{k-1}, z_k],C''_{z_k})$ of   liftings  $[z_{k-1}, z_k]$ of $[p_{k-1}, p_k]$ and $C''_{z_k}$ of ${C^*_{p_k}}$ such that  $[z_k, z_{k-1})$ is opposite to    $[z_k, z_{k+1})$ (see Lemma \ref {sc6d}),  $[z_k, z_{k-1})\in \ov{C''_{z_k}}$ and $C''_{z_k}$ is the decoration of    $[z_k, z_{k-1}]$ associated to $C_y$.     Let us consider an apartment   $A$ containing $C_{x}$ and $C''_{z_{k+1}}$    hence also $[z_k, z_{k+1}]$ and $C_{z_{k+1}}^+$ (see Lemma \ref{sc6b}.5). 
The restriction $\qr\vert_{A}$ is the restriction to $A$ of an automorphism $\varphi$ of $\SHI$ fixing $C_{x}$ that induces an isomorphism $\varphi\vert_{\sht_{z_k}\SHI}$ from ${\sht_{z_k}\SHI}$ onto ${\sht_{p_k}\SHI}$. So the number of liftings   $\un{[z_{k-1}, z_k]}$ is the same as in the case $z_k=p_k$ \ie the number of liftings   $\un{[z_{k-1}, p_k]}$ of  $\un{[p_{k-1}, p_k]}$ such that $[p_k,z_{k-1})$ is opposite to $\qr([z_k, z_{k+1}])=[p_k, p_{k+1}]$ and $\qr (C''_{z_k})=\qr_{\A_{p_k}, {C_{p_k}^-} } \circ \varphi\vert_{\sht_{z_k}\SHI}(C''_{z_k})=\qr_{\A, C_{p_k}^-}(C''_{p_k}) =C_{p_k}^*$. 

By Lemma \ref{sc6c} the possible $([p_k, z_{k-1}), C''_{p_k})$ (and so the possible $([p_k, z_{k-1}], C''_{p_k})$ by Lemma \ref{sc6b}.5) are in one-to-one correspondance with the union 
the sets $\mathcal C^m_{C_{p_k}^+}(C_{p_k}^-,\mathbf c)$ for $\mathbf c$ in the set $\Gamma_{C_{p_k}^+}^+(C_{p_k} ^-,\mathbf i_\ell,\tilde C_{p_k})$, with $\tilde C_{p_k}=pr_{-\eta_k}(C_{p_k}^+)$.

\par 3) For the last step of the lifting, by the same way as before, we suppose given the lifting \green{of} $\un{[z_0, y]}$ and we suppose $z_0=p_0$. So we know that $C_{p_0}^+=C_{z_0}^+$. 
The Lemma \ref{sc6} says that  $d^{*W}(C_{p_0}^-,C_{z_0})=w^+_{\ql}w$, and  Lemma \ref{sc6aa} that $d^W((C_{p_0}^+,C_{z_0})=w_{v^{-1}\mu }v^{-1}$. So, as before, the number $C_{z_0}$ is the number of elements of the different sets $\mathcal C^m_{C_{p_0}^-}(C_{p_0}^+,\mathbf e)$ where $\mathbf e$  is a gallery element of  $\Gamma_{C_{p_0}^-}^+(C_{p_0} ^+,\mathbf i,C'_{p_0})$ as $\mathbf i$ is the type of a minimal gallery from $C_{p_0}^+$  to  $C_{z_0}$ that retracts by $\qr_{\A, C_{p_0}^-}$ to a gallery from $C_{p_0}^+$  to  $C'_{p_0}$.
\end{proof} 

\subsection{Consequence}\label{sc8} 

\par The above explicit formula, together with the formula for $\sharp \mathcal C^m_{\mathfrak Q}(C^-_{z},\mathbf c)$ in \ref{sc2}, tell us that the structure constant $a_{\mathbf w,\mathbf v}^{\mathbf u} $ is a polynomial in the parameters $q_{i}-1,q'_{i}-1$ for $q_i,q'_i\in\shq$ with coefficients in $\N=\Z_{\geq0}$ and that this polynomial depends only on $\A$, $W$, $\mathbf w$, $\mathbf v$ and $\mathbf u$.
So we have proved the conjecture 1 of the introduction in this generic case: when $\ql$ and $\mu$ are spherical. 

\par Note that we have not got all the structure constants $a_{\mathbf w,\mathbf v}^{\mathbf u} $ for the generic Iwahori-Hecke algebra $^I\mathcal H_\Z^g$.
The cases $\mathbf w\in W^v\ltimes V_{0}$ or $\mathbf v\in W^v\ltimes V_{0}$ (\ie $\ql\in V_{0}$ or $\qm\in V_{0}$ in the above notations) are missing.
We deal with them in the following section.

\section{Structure constants in remaining generic cases}\label{s4}

\subsection{The problem}\label{4.1} 

\par Let us choose $C_x,C_y\in \mathscr C_0^+$ with $x\leq y$ and $d^W(C_x,C_y) = \mathbf u=\qn.u\in W^+=W^v\ltimes Y^+$. 
Then  the structure constant $a_{\mathbf w,\mathbf v}^{\mathbf u}$
(for $\mathbf w=\ql.w$ and $\mathbf v=\qm.v$ in $W^+$)  is the number of $C_{z_{0}}\in \mathscr C_0^+$ with $x\leq {z_{0}}\leq y$, $d^W(C_x,C_{z_{0}}) = \mathbf w$ and $d^W(C_{z_{0}},C_y) = \mathbf v$, see Proposition \ref{PrFinite2}.

\par In Theorem \ref{sc7}, we computed $a_{\mathbf w,\mathbf v}^{\mathbf u}$ when $\mathbf w,\mathbf v$ are spherical (\ie $\ql,\qm\in Y\cap\sht^\circ$).
We shall compute it below in the remaining cases where $\mathbf w,\mathbf v\in W^ {+g}=W^v\ltimes(Y\cap(\sht^\circ\cup V_{0}))$.
So, in the affine or strictly hyperbolic cases, we shall get $a_{\mathbf w,\mathbf v}^{\mathbf u}$ for any $\mathbf w,\mathbf v\in W^+$.
But we get, in general, these structure constants for $\mathbf w,\mathbf v\in W^{+g}=W^v\ltimes Y^ {+g}$, \ie we get the structure constants of $^I\mathcal H^g$, see \ref{sc8} and \ref{4.6}.

\par We start with a lemma analogous to lemmas \ref{sc6} and \ref{sc6aa}.

\begin{lemm}\label{4.2}  
Let $C_x, C_z \in \mathscr C_0^+$ with $x\leq z$ and $\ql \in Y^{+0}$ , $w\in W^v$.
We write $C^+_x=pr_x(C_z)$, then
$$
d^W(C_x,C_z) = \lambda . w \Longleftrightarrow
\left \{
\begin{array}{l}
d^W(C_x,z) = \lambda \\
d^{W}(C^-_z,C_z) = w .
\end{array}
\Longleftrightarrow
\left \{
\begin{array}{l}
d^W(C_x,z) = \lambda \\
d^{W}(C_{x},C^+_{x}) =  w .
\end{array}
\right.
\right.
$$
\par Actually $d^W(C_x,z) = \lambda\in V_{0}$ implies $x\leq z$ and $z\leq x$.
So $C^-_{z}:=pr_{z}(C_{x})$ is well defined, by \ref{sc0}.1, and is a positive local chamber.
\end{lemm}

\begin{proof} By definition $d^W(C_x,C_z) = \lambda . w$ implies $d^W(C_x,z) = \lambda$ (\ref{1.13}).
Suppose now $d^W(C_x,z) = \lambda$. Then $d^v(x,z)=\ql\in V_{0}$, so any apartment $A$ containing $x$ or $z$ contains $z$ or $x$ and, in $A$, one has $z=x+\ql\leq x$; this is a consequence of \ref{1.3}.1.a, as any enclosure is stable under $V_{0}$.
Hence $C^-_{z}=pr_{z}(C_{x})\in A$ is well defined, by \ref{sc0}.1, and  a positive local chamber.
Actually $C^-_{z}=C_{x}+\ql$ (calculation in $A$).
We have also $C^+_{x}=C_{z}-\ql$.
It is now clear that $d^W(C_x,C_z) = \lambda . w  \iff d^{W}(C^-_z,C_z) = w \iff d^{W}(C_{x},C^+_{x}) =  w $.
\end{proof}

\subsection{First reduction}\label{4.3} 

\par We consider $\mathbf u,\mathbf v,\mathbf w \in W^+$ and write $\mathbf u=\qn.u,\mathbf v=\qm.v,\mathbf w=\ql.w$ with $\ql,\qm,\qn\in Y^+$ and $u,v,w\in W^v$.
We choose $C_x,C_y\in \mathscr C_0^+$ with $x\leq y$ and $d^W(C_x,C_y) = \mathbf u$; we may suppose $C_{x},C_{y}\subset\A$. 
We choose $C_{z_{0}}\in \mathscr C_0^+$ with $x\leq {z_{0}}\leq y$, $d^W(C_x,C_{z_{0}}) = \mathbf w$ and $d^W(C_{z_{0}},C_y) = \mathbf v$.

\par If $\ql\in Y^ {+0}=Y\cap V_{0}$, one has $d^W(C_{x},z_{0})=\ql$ (Lemma \ref{4.2}) and  $z_{0}\in\A$, more precisely $z_{0}=x+\ql$ (as we saw in the proof of Lemma \ref{4.2}).

\par If $\qm\in Y^ {+0}$, then we get  $z_{0}\in\A$, more precisely $z_{0}=y-\qm$, by Lemma \ref{4.2} applied to $C_{z_{0}},C_{y}$ instead of $C_{x},C_{z}$.

\par In both cases $z_{0}$ has to be a well determined point in $\A$  and $\qn=d^v(x,y) \in W^v\ql + W^v\qm$.
In particular, if $\mathbf w,\mathbf v \in W^ {+g}$ \ie $\ql,\qm  \in Y^ {+g}$, one has also $\qn  \in Y^ {+g}$ \ie $\mathbf u \in W^ {+g}$.

\par  We want now to compute the number $a_{\mathbf w,\mathbf v}^{\mathbf u}$ of $C_{z_{0}}\in \mathscr C_0^+$ with $x\leq {z_{0}}\leq y$, $d^W(C_x,C_{z_{0}}) = \mathbf w$ and $d^W(C_{z_{0}},C_y) = \mathbf v$. For this
we separate below the cases $\ql\in Y^ {+0}$ and $\qm\in Y^ {+0}$.

\subsection{The case $\qm\in Y^ {+0}$}\label{4.4} 

\par We suppose $\ql\in Y\cap \sht^\circ$ (\resp $\ql\in Y^ {+0}$).
By Lemma \ref{4.2} above and Lemma \ref{sc6}, we have to find the number $a_{\mathbf w,\mathbf v}^{\mathbf u}$ of $C_{z_{0}}\in \mathscr C_0^+$ satisfying:

\par (a) $d^W(C_{x},z_{0})=\ql$ ,\qquad (b) $d^W(C_{z_{0}},y)=\qm$ , \qquad (c) $d^W(C_{z_{0}},C^+_{z_{0}})=v$

\parni   and  (d) $d^ {*W}(C^-_{z_{0}},C_{z_{0}})=w^+_{\ql}.w$ \qquad (\resp and (d) $d^ {W}(C^-_{z_{0}},C_{z_{0}})=w$).

\par Actually $\qm\in V_{0}$ is fixed by $W^v$ and $y,C_{z_{0}},C^+_{z_{0}}$ are in a same apartment (containing $C_{y}$ and $C_{z_{0}}$), so $d^W(C_{z_{0}},y)=\qm \iff d^W(C^+_{z_{0}},y)=\qm$.
Then $a_{\mathbf w,\mathbf v}^{\mathbf u}$ is the number of $C_{z_{0}}\in \mathscr C_0^+$ satisfying (a), (b') $d^W(C^+_{z_{0}},y)=\qm$, (c) and (d).
The first two conditions involve only $z_{0},C_{x},C_{y}\in\A$.

\begin{prop*} The number $a_{\mathbf w,\mathbf v}^{\mathbf u}$ is either $0$ (if the conditions (a), (b') above are incompatible) or 

\qquad\qquad\qquad\qquad $\sum _{\mathbf e\in\Gamma_{C^-_{z_0}}^+(C^+_{z_{0}},\mathbf i,C'_{z_0})} \sharp \mathcal C^m_{C^-_{z_0}} (C^+_{z_{0}},\mathbf e)$ 

where $\mathbf i$ is the type of a fixed reduced decomposition of $v^ {-1}$ and $C'_{z_0}$ is the unique local chamber at $z_{0}$ in $\A$ such that $d^ {*W}(C^-_{z_{0}},C'_{z_{0}})=w^+_{\ql}.w$  (\resp $d^ {W}(C^-_{z_{0}},C'_{z_{0}})=w$).
\end{prop*}

\begin{rema*} The coefficient $a_{\mathbf w,\mathbf v}^{\mathbf u}$ is zero when (a) and (b') are incompatible, \ie when $\qn \neq \ql+\qm$: if in $\A$ we identify $C_{x}$ to the fundamental chamber $C^+_{0}$, (a) is equivalent to $z_{0}=x+\ql$, (b') to $y=z_{0}+\qm$ and $d^W(C_{x},C_{y})=\qn.u$ implies $y=x+\qn$.

\par But the other case where $a_{\mathbf w,\mathbf v}^{\mathbf u}=0$ is when $\Gamma_{C^-_{z_0}}^+(C^+_{z_{0}},\mathbf i,C'_{z_0})$ is empty.
\end{rema*}

\begin{proof} We have to translate the conditions (c) and (d).
We consider the retraction $\qr=\qr_{\A,C^-_{z_{0}}}$.
The condition (c) is equivalent to the existence of a minimal gallery $\mathbf c$ starting from $C^+_{z_{0}}$, of type $\mathbf i$ (\ie $\mathbf c\in \mathcal C^m(C^+_{z_{0}},\mathbf i)$) ending in $C_{z_{0}}$; and there is a bijection between these $\mathbf c$ and the $C_{z_{0}}$ satisfying (c).
Now the condition (d) is equivalent to $\qr(C_{z_{0}})=C'_{z_{0}}$ (as $\qr$ preserves the $W-$distances to $C^-_{z_{0}})$.
Considering $\mathbf e=\qr(\mathbf c)$, the proposition is now clear.
\end{proof}

\subsection{The case $\ql\in Y^ {+0}$ (and $\qm\in Y\cap \sht^\circ$)}\label{4.5} 

\par By Lemma \ref{4.2} above and Lemma \ref{sc6aa}, we have to find the number $a_{\mathbf w,\mathbf v}^{\mathbf u}$ of $C_{z_{0}}\in \mathscr C_0^+$ satisfying:

\par (a) $d^W(C_{x},z_{0})=\ql$ ,\qquad (b) $d^W(C^+_{z_{0}},y)=\qm^ {++}$ ,
\qquad (c) $d^ {*W}(C''_{y},C_y)=w^+_{\mu^{++}}w_{v^{-1}.\mu}$

\par (d) $d^ {W}(C^-_{z_{0}},C_{z_{0}})=w$
\quad  and \quad (e) $d^ {W}(C^+_{z_{0}},C_{z_{0}})=w_{v^ {-1}\qm}.v^ {-1}$ 

\parni But $C^+_{z_{0}}=pr_{z_{0}}(C_{y})$, $C''_{y}=pr_{y}(C^+_{z_{0}})$ and $C_{x},C_{y},z_{0}=x+\ql$ are in $\A$.
So the conditions (a), (b), (c) involve only $C_{x}, C_{y}$ and $z_{0}$.

\begin{prop*} The number $a_{\mathbf w,\mathbf v}^{\mathbf u}$ is either $0$ (if the conditions (a), (b), (c) above are incompatible) or 

\qquad\qquad\qquad\qquad $\sum _{\mathbf e\in\Gamma_{C^-_{z_0}}^+(C^+_{z_{0}},\mathbf i,C'_{z_0})} \sharp \mathcal C^m_{C^-_{z_0}} (C^+_{z_{0}},\mathbf e)$ 

where $\mathbf i$ is the type of a fixed reduced decomposition of $w_{v^ {-1}\qm}.v^ {-1}$ and $C'_{z_0}$ is the unique local chamber at $z_{0}$ in $\A$ such that  $d^ {W}(C^-_{z_{0}},C'_{z_{0}})=w$.
\end{prop*}

\begin{rema*} The coefficient $a_{\mathbf w,\mathbf v}^{\mathbf u}$ is zero when (a), (b) and (c) are incompatible, \ie when $z_{0}$, determined by (b) does not satisfy (a) and (c).
But it is more difficult than in \ref{4.4} to translate it simply.
It is also zero when $\Gamma_{C^-_{z_0}}^+(C^+_{z_{0}},\mathbf i,C'_{z_0})$ is empty.
\end{rema*}

\begin{proof} We have to translate conditions (d) and (e). It goes the same way as in \ref{4.4}.
\end{proof}

\subsection{Conclusion}\label{4.6} 

\par In all cases where $\ql,\qm\in Y^ {+g}=Y\cap(\sht^\circ \cup V_{0})$, we may use the formula for $ \mathcal C^m_{\g Q} (C'_{z},\mathbf c)$ in \ref{sc2}, the Theorem \ref{sc7} and/or the Propositions \ref{4.4}, \ref{4.5}.
We get the expected result: the structure constant $a_{\mathbf w,\mathbf v}^{\mathbf u} $ is a polynomial in the parameters $q_{i}-1,q'_{i}-1$ for $q_i,q'_i\in\shq$ with coefficients in $\N=\Z_{\geq0}$ and  this polynomial depends only on $\A$, $W$, $\mathbf w$, $\mathbf v$ and $\mathbf u$.
We have proved Conjecture 1 in these cases, in particular in the affine or strictly hyperbolic cases.


\medskip
\par\noindent Universit\'e de Lorraine, CNRS, IECL, F-54000 Nancy, France

E-mail: Nicole.Panse@univ-lorraine.fr ; Guy.Rousseau@univ-lorraine.fr

\end{document}